\documentclass[11pt]{amsart}
\usepackage[T1]{fontenc}
\usepackage[utf8]{inputenc}
\usepackage{lmodern}
\usepackage{microtype}
\usepackage{xcolor}
\usepackage{amsmath,amssymb,amsthm,mathtools}
\usepackage{enumitem}
\usepackage[margin=1in]{geometry}
\usepackage{hyperref}
\hypersetup{colorlinks=true,linkcolor=blue,citecolor=blue,urlcolor=blue}
\usepackage{ifthen}

\theoremstyle{plain}
\newtheorem{theorem}{Theorem}[section]
\newtheorem{lemma}[theorem]{Lemma}
\newtheorem{proposition}[theorem]{Proposition}

\theoremstyle{definition}
\newtheorem{definition}[theorem]{Definition}
\newtheorem{example}[theorem]{Example}
\theoremstyle{remark}
\newtheorem{remark}[theorem]{Remark}

\newlength{\leftstackrelawd}
\newlength{\leftstackrelbwd}
\def\leftstackrel#1#2{\settowidth{\leftstackrelawd}{${{}^{#1}}$}\settowidth{\leftstackrelbwd}{$#2$}
	\addtolength{\leftstackrelawd}{-\leftstackrelbwd}
	\leavevmode\ifthenelse{\lengthtest{\leftstackrelawd>0pt}}
	{\kern-.5\leftstackrelawd}{}\mathrel{\mathop{#2}\limits^{#1}}}

\DeclareMathOperator{\rank}{rank}
\DeclareMathOperator{\diag}{diag}
\newcommand{\DD}{\mathbb{D}}
\newcommand{\FF}{\mathbb{F}}
\newcommand{\R}{\mathcal{R}}

\newcommand{\GL}{\mathrm{GL}}
\newcommand{\eps}{\varepsilon}
\newcommand{\Cent}{\operatorname{Cent}}
\newcommand{\supp}{\operatorname{supp}}
\numberwithin{equation}{section}

\newcommand{\abs}[1]{\left|#1\right|}

\title[Jordan semi-triple maps on structural matrix rings]{Automatic additivity for injective Jordan semi-triple maps on structural matrix rings over division rings}

\author{Ilja Gogi\'{c}}
\address{I.~Gogi\'{c}, Department of Mathematics, Faculty of Science, University of Zagreb, Bijeni\v{c}ka 30, 10000 Zagreb, Croatia}
\email{ilja@math.hr}

\author{Mateo Toma\v{s}evi\'{c}}
\address{M.~Toma\v{s}evi\'{c}, Department of Mathematics, Faculty of Science, University of Zagreb, Bijeni\v{c}ka 30, 10000 Zagreb, Croatia}
\email{mateo.tomasevic@math.hr}

\subjclass[2020]{16W20, 16S50, 17C50, 47B49}

\keywords{Jordan semi-triple map, Jordan homomorphism, structural matrix ring, division ring, automatic additivity, nonlinear preserver}

\date{\today}

\begin{document}
	
	\begin{abstract}
		Let $\mathbb D$ be a division ring, and let
		$\mathcal{R}\subseteq M_n(\mathbb{D})$ be a structural matrix ring over
		$\mathbb{D}$, that is, the subring of $M_n(\mathbb{D})$ supported on the
		ordered pairs of a preorder on $\{1,\ldots,n\}$. We study injective Jordan
		semi-triple maps $\phi:\mathcal{R}\to M_n(\mathbb{D})$, namely injective maps
		satisfying
		\[
		\phi(XYX)=\phi(X)\phi(Y)\phi(X), \qquad \text{for all } X,Y\in\mathcal{R}.
		\]
		Assuming that the centre of $\mathbb{D}$ has more than two elements, we give a
		criterion for automatic additivity and show that there are exactly two
		obstructions. The first one is scalar: it occurs precisely when
		$\mathcal{R}$ has a direct ring summand isomorphic to $\mathbb{D}$ and
		$\mathbb{D}$ is isomorphic to neither $\mathbb{F}_3$ nor $\mathbb{F}_4$. The second one is order-theoretic: it occurs when a nonsymmetric comparable pair $i\preceq j$, $j\not\preceq i$, admits
		no third index $k\notin\{i,j\}$ comparable with both $i$ and $j$. If neither
		obstruction occurs, all injective Jordan semi-triple maps are additive. The centre-size hypothesis is sharp: for $n\ge3$, the upper-triangular ring $T_n(\mathbb{F}_2)$ has neither obstruction but nevertheless admits nonadditive injective Jordan semi-triple maps. Finally, in the additive case,
		we describe the maps componentwise, in terms of endomorphisms,
		anti-endomorphisms, and transitive multipliers. 
	\end{abstract}
	
	\maketitle
	
	\section{Introduction}
	Let $\mathcal R$ and $\mathcal S$ be rings. A map $\phi:\mathcal R\to\mathcal S$ is called a \emph{Jordan semi-triple map}
	if it satisfies
	\[
	\phi(xyx)=\phi(x)\phi(y)\phi(x),
	\qquad \text{for all } x,y\in\mathcal R.
	\]
	No additivity is assumed; the question is when the identity and
	injectivity already force it. The terminology follows Le\v{s}njak and
	Sze~\cite{Lesnjak-Sze}, who proved automatic additivity and obtained a
	standard form for full matrix algebras over fields. We extend the
	full-matrix setting to structural matrix rings over division rings, where
	additional scalar and order-theoretic obstructions arise. Further related results on Jordan semi-triple maps can be found in \cite{DAmour,Kokol-Mojskerc,Kuzma-JTP,Lu-triple,Slowik,Yang-Lu}.
	
	The semi-triple identity also has a natural Jordan-theoretic origin. Let
	$\mathcal J$ be a Jordan algebra over a field of characteristic different
	from $2$, with product $\circ$. The quadratic operator associated with
	$x\in\mathcal J$ is
	\[
	U_x:\mathcal J\to\mathcal J,\qquad
	U_x(y):=2x\circ(x\circ y)-x^2\circ y;
	\]
	see, for example, \cite{Jacobson-Jordan,McCrimmon-Jordan}. When
	$\mathcal J$ arises from an associative algebra with the symmetrized
	product
	\[
	x\circ y=\frac12(xy+yx),
	\]
	one has $U_x(y)=xyx$. Thus, in the associative case, the semi-triple
	identity expresses the preservation of quadratic operators.
	
	The problem considered here belongs to the broader automatic-additivity
	programme: one starts with a map satisfying a multiplicative, Jordan
	multiplicative, ternary, or similar nonlinear identity, and asks whether
	additivity follows. A classical theorem of Martindale~\cite{Martindale} shows that bijective multiplicative maps from prime rings with nontrivial idempotents onto arbitrary rings are
	automatically additive. In the matrix setting, Jodeit and
	Lam~\cite{Jodeit-Lam} treated nondegenerate multiplicative self-maps over
	principal ideal domains. Similar phenomena for Jordan multiplicative maps and related nonlinear
	preservers have been studied in full matrix algebras, standard operator algebras, and Jordan-algebraic settings; see, for example,
	\cite{GT-matrix-jordan-mult,Ji,Molnar}.
	
	Let $\preceq$ be a preorder (or quasi-order) on
	$[n]:=\{1,\ldots,n\}$. The associated \emph{structural matrix ring},
	abbreviated \emph{SMR}, over a division ring $\DD$ is
	\[
	\R_{\preceq}
	:=
	\{[x_{ij}]\in M_n(\DD):x_{ij}=0
	\text{ whenever } i\not\preceq j\}.
	\]
	Thus $\R_{\preceq}$ is the unital subring of $M_n(\DD)$ whose zero
	pattern is encoded by $\preceq$. This terminology is standard in the
	ring-theoretic literature; see, for instance, \cite{Coelho-CA,Smith-vanWyk-internal,vanWyk-special-radicals,vanWyk-column-sum}. This class contains full matrix rings, triangular and block triangular matrix rings, and finite incidence rings, thereby extending Rota's incidence-algebra framework~\cite{Rota}. Equivalently, as recalled in Remark~\ref{rem:diagonal-characterization}, the SMRs over $\DD$ are precisely the subrings of $M_n(\DD)$ that contain the full diagonal subring. Over fields, these rings are usually called \emph{structural matrix algebras}, or \emph{SMAs}.
	
	SMRs and SMAs have been studied extensively from several points of view.
	Automorphisms of SMRs and SMAs were described by
	Coelho~\cite{Coelho-LAA,Coelho-CA}. Jordan homomorphisms and Jordan
	isomorphisms, including the incidence ring setting, were studied in
	\cite{Akkurt-Akkurt-Barker,Benkovic, BFK-finitary,BFK-pocategory}.
	Further preserver problems for incidence algebras, such as potent
	preservers and primitive-idempotent preservers, were considered in
	\cite{Garces-Khrypchenko-potent,Garces-Khrypchenko-primitive}. More recently, the authors studied SMAs over the complex field, describing Jordan embeddings and related rank preservers~\cite{GT-Jordan}
	and obtaining automatic-additivity criteria and standard forms for
	multiplicative and Jordan multiplicative maps~\cite{GT-mult}. The transitive multipliers appearing in
	these standard forms will reappear below.
	
	For an arbitrary SMR $\R_{\preceq}\subseteq M_n(\DD)$, we determine when every
	injective Jordan semi-triple map $\phi:\R_{\preceq}\to M_n(\DD)$ is additive.
	The main result, Theorem~\ref{thm:main}, gives a necessary and sufficient
	criterion when the centre $Z(\DD)$ has more than two elements. There are
	precisely two possible obstructions.
	
	The first obstruction is scalar. It appears when some index $i$ is
	incomparable with all the others. In that case, $\DD E_{ii}$ is a direct
	ring summand of $\R_{\preceq}$, isomorphic to $\DD$, so the problem reduces
	to the scalar case on this summand. By
	Proposition~\ref{prop:scalar-classification}, automatic additivity in the
	scalar case holds exactly for $\DD\cong\FF_2,\FF_3$ or $\FF_4$. Under the
	standing hypothesis $|Z(\DD)|>2$, this leaves precisely the two scalar
	exceptions $\DD\cong\FF_3$ and $\DD\cong\FF_4$.
	
	The second obstruction is order-theoretic. Writing $E_{ij}$ for the standard matrix
	units, for pairwise distinct indices $i,j,k$ with $i\preceq j\preceq k$
	and $a,b\in\DD$ we have
	\[
	(aE_{ij}+bE_{jk})E_{jj}(aE_{ij}+bE_{jk})=abE_{ik}.
	\]
	This identity links three matrix components, with the coefficient order
	fixed by multiplication in $\DD$. After relabelling the indices, the same
	mechanism applies whenever a third index lies below $i$, between $i$ and
	$j$, or above $j$. We call a nonsymmetric comparable pair
	$i\preceq j$, $j\not\preceq i$, \emph{isolated} if no such third index
	exists.  In that case, one can deform the triangular corner spanned by
	$E_{ii},E_{ij},E_{jj}$ by modifying the $E_{ij}$-coordinate nonadditively,
	while still preserving the Jordan semi-triple identity.
	
	Thus, when $|Z(\DD)|>2$, automatic additivity holds precisely when no
	isolated nonsymmetric comparable pair exists and, whenever a scalar direct
	summand occurs, one has $\DD\cong\FF_3$ or $\FF_4$. The sufficiency
	direction requires only $|\DD|>2$, while the stronger centre-size hypothesis
	is used in the reverse implication. The role and limits of this centre-size
	hypothesis are discussed in Section~\ref{sec:centre-size-hypothesis}.
	
	Independently of the automatic-additivity criterion, we also describe
	additive injective Jordan semi-triple maps on arbitrary SMRs. After
	normalization, such maps have a standard componentwise form: on each
	connected component, the action is determined either by an injective
	endomorphism of $\DD$ or by an injective anti-endomorphism of $\DD$, together
	with a transitive multiplier.
	
	The paper is organized as follows. Section~\ref{sec:prelim-obstructions} fixes notation and recalls the structural facts about SMRs used later. Section~\ref{sec:obstructions-main}
	introduces the two obstructions to automatic additivity and states
	Theorem~\ref{thm:main}. Section~\ref{sec:proof-main} proves the theorem.
	Section~\ref{sec:centre-size-hypothesis} discusses the role of the
	centre-size hypothesis. Finally, Section~\ref{sec:standard-form}
	establishes the standard form for additive injective Jordan semi-triple maps.
	
	\section{Preliminaries}\label{sec:prelim-obstructions}
	
	Throughout, $\DD$ denotes a division ring, $Z(\DD)$ its centre, and
	$\DD^\times:=\DD\setminus\{0\}$. By $M_{m,n}(\DD)$ we denote the set of all
	$m\times n$ matrices over $\DD$, equipped with its usual left $\DD$-vector
	space structure. When $m=n$, we write $M_n(\DD)$ for the corresponding matrix
	ring, with the usual matrix multiplication. We denote by $T_n(\DD)$ the
	subring of $M_n(\DD)$ consisting of all upper-triangular matrices.
	
	We view $\DD^n$ as a right $\DD$-vector space of column vectors, and matrices
	act on it by left multiplication. Thus, for $A\in M_{m,n}(\DD)$, we write
	\[
	\rank A:=\dim_{\DD}^{\,r} A(\DD^n),
	\]
	the right $\DD$-dimension of the image. Whenever rank is mentioned, it is meant
	in this sense. 
	
	When working in $M_n(\DD)$, we write $I$ for the identity matrix, or $I_n$
	when we wish to emphasize its size. Matrix-unit coefficients are written on the left. Hence, for $a,b\in\DD$,
	\begin{equation}\label{eq:matrix-unit-product}
		(aE_{ij})(bE_{k\ell})=
		\begin{cases}
			abE_{i\ell},& \text{ if } j=k,\\
			0, & \text{ if } j\ne k.
		\end{cases}
	\end{equation}
	The order of the coefficients in \eqref{eq:matrix-unit-product} is important
	when $\DD$ is noncommutative. In particular, if
	$D=\diag(d_1,\ldots,d_n)$ with $d_1,\ldots,d_n\in\DD^\times$, then
	\[
	D(cE_{ij})D^{-1}=d_i c d_j^{-1}E_{ij}.
	\]
	Following our previous work, for a matrix
	$X=[x_{ij}]\in M_n(\DD)$, we define its \emph{support} as
	\[
	\supp X :=\{(i,j)\in[n]\times[n]\colon x_{ij}\ne0\}.
	\]
	If $\Omega\subseteq[n]\times[n]$, we also say that $X$ is \emph{supported in}
	$\Omega$ if $\supp X\subseteq\Omega$. For $S\subseteq[n]$ and $X\in M_n(\DD)$, set
	\[
	P_S:=\sum_{i\in S}E_{ii}.
	\]
	Note that $X$ is supported in $S \times S$ precisely when $X=P_SXP_S$.
	
	Given a preorder $\preceq$ on $[n]$, we write
	\[
	i\prec j
	\quad\stackrel{\text{def}}\iff\quad
	i\preceq j\ \text{and}\ i\ne j.
	\]
	Since $\preceq$ need not be antisymmetric, $i\prec j$ does not exclude
	$j\preceq i$. We also write $\asymp$ for the symmetric part of $\preceq$, that is,
	\[
	i\asymp j
	\quad\stackrel{\text{def}}\iff\quad
	i\preceq j\ \text{and}\ j\preceq i.
	\]
	Then $\asymp$ is an equivalence relation on $[n]$.
	
	Following van Wyk \cite{vanWyk-special-radicals}, an SMR obtained from a preorder $\preceq$ on $[n]$ is defined as
	\[
	\R_{\preceq}:=\bigoplus_{i\preceq j}\DD E_{ij}\subseteq M_n(\DD),
	\]
	where the direct sum is taken as a left $\DD$-vector space. As in the field case for structural matrix algebras \cite[Proposition~3.1]{GT-Jordan}, we record the following elementary
	characterization.
	
	\begin{remark}\label{rem:diagonal-characterization}
		A subring $\R\subseteq M_n(\DD)$ is an SMR if and only if it contains the full
		diagonal subring $\bigoplus_i\DD E_{ii}$. Indeed, suppose that
		$\bigoplus_i\DD E_{ii}\subseteq\R$, and define a relation $\preceq_{\R}$ on $[n]$ by
		\[
		i\preceq_{\R}j
		\quad\stackrel{\text{def}}\Longleftrightarrow\quad
		\DD E_{ij}\subseteq\R .
		\]
		Since $\R$ contains all diagonal matrices, the relation $\preceq_{\R}$ is
		clearly reflexive. It is also transitive: if
		$i\preceq_{\R}j\preceq_{\R}k$, then $aE_{ik}=(aE_{ij})E_{jk}\in\R$ for all $a \in \DD$. Thus $\preceq_{\R}$ is a preorder. Let $\mathcal S\subseteq M_n(\DD)$ be the SMR defined by $\preceq_{\R}$. By definition, $\mathcal S\subseteq\R$. Conversely, if $A=[a_{ij}]\in\R$ and $a_{ij}\ne0$, then, for every $a\in\DD$,
		\[
		aE_{ij}=(aa_{ij}^{-1}E_{ii})AE_{jj}\in\R.
		\]
		Thus every nonzero coordinate of $A$ belongs to the support of $\mathcal S$,
		so $A\in\mathcal S$. Hence $\R=\mathcal S$, and therefore $\R$ is an SMR. The
		converse is immediate from the definition.
	\end{remark}
	
	Let $\Gamma_{\preceq}$ denote the \emph{comparability graph} of a preorder $\preceq$ on $[n]$.
	Its vertex set is $[n]$, and two distinct vertices $i$ and $j$ are joined by
	an edge precisely when they are comparable, that is, when $i\preceq j$ or
	$j\preceq i$. The following lemma is the SMR variant of \cite[Remark~3.3]{GT-Jordan}.
	
	\begin{lemma}\label{lem:central-components}
		Let $C_1,\ldots,C_m$ be the connected components of $\Gamma_{\preceq}$. Then
		$P_{C_1},\ldots,P_{C_m}$ are pairwise orthogonal central idempotents of
		$\R_{\preceq}$ satisfying $\sum_{k=1}^m P_{C_k}=I$, and $\R_{\preceq}$ decomposes as the direct sum
		of the corresponding two-sided ideals:
		\begin{equation}\label{eq:PC-direct-sum}
			\R_{\preceq}
			=
			\bigoplus_{k=1}^m P_{C_k}\R_{\preceq}P_{C_k}.
		\end{equation}
		Moreover, every central idempotent of $\R_{\preceq}$ is a sum of some of the
		$P_{C_k}$'s. Consequently, $P_{C_1},\ldots,P_{C_m}$ are precisely the minimal
		nonzero central idempotents of $\R_{\preceq}$, with respect to the usual partial order on idempotents, where $P\le Q$ means $PQ=QP=P$.
	\end{lemma}
	
	\begin{proof}
		Let $C$ be a connected component of $\Gamma_{\preceq}$. If
		$i \preceq j$, then either both $i$ and $j$ lie in $C$, or neither
		does. Hence,
		for every $a\in\DD$ and every $E_{ij}\in\R_{\preceq}$, we have
		\[
		P_C(aE_{ij})=
		\begin{cases}
			aE_{ij},& i,j\in C,\\
			0,& i,j\notin C
		\end{cases}
		=
		(aE_{ij})P_C.
		\]
		Therefore $P_C$ commutes with every element of $\R_{\preceq}$, and so
		$P_C$ is central. Since the connected components
		partition $[n]$, the idempotents $P_{C_1},\ldots,P_{C_m}$ are pairwise
		orthogonal and sum to $I$. This gives the decomposition \eqref{eq:PC-direct-sum}.
		
		Now let $Q$ be a central idempotent of $\R_{\preceq}$. Since
		$Q$ commutes with every $E_{ii}$, we have $E_{ii}QE_{jj}=0$ for $i\ne j$.
		Thus $Q=\sum_{i=1}^n d_iE_{ii}$ for some $d_i \in \DD$. The equality $Q^2=Q$ gives $d_i^2=d_i$,
		whence $d_i=0$ or $d_i=1$. If $i\preceq j$,
		then centrality with respect to $E_{ij}$ gives $d_iE_{ij}=d_jE_{ij}$, and
		hence $d_i=d_j$. Thus $d_i=d_j$ whenever vertices $i$ and $j$ are adjacent in
		$\Gamma_{\preceq}$. Since any two vertices in the same connected component of
		$\Gamma_{\preceq}$ are joined by a path of adjacent vertices, these equalities
		propagate along the path. Therefore the coefficients $d_i$ are constant on
		connected components. It follows that $Q$ is a sum of
		the corresponding $P_{C_k}$'s. Hence, the $\le$-minimal nonzero central idempotents of $\R_{\preceq}$ are
		exactly $P_{C_1},\ldots,P_{C_m}$.
	\end{proof}
	
	\section{Jordan semi-triple maps: obstructions and main result}\label{sec:obstructions-main}
	
	We begin this section by recalling the notion of a Jordan semi-triple map.
	We then isolate the two phenomena responsible for failure of automatic
	additivity: nonadditive scalar Jordan semi-triple maps and isolated
	nonsymmetric comparable pairs in the underlying preorder. Theorem~\ref{thm:main},
	the main result of the paper, shows that these are the only possibilities.
	
	\begin{definition}\label{def:jst}
		Let $\mathcal{R}$ and $\mathcal{S}$ be rings. A map $\phi:\mathcal{R}\to\mathcal{S}$ is a \emph{Jordan semi-triple map} if it satisfies
		\[
		\phi(xyx)=\phi(x)\phi(y)\phi(x), \qquad \text{for all } x,y\in\mathcal{R}.
		\]
		When $\mathcal{R}=\mathcal{S}=\DD$, we call $\phi$ a \emph{scalar Jordan semi-triple map}.
	\end{definition}
	
	We now focus on obstructions to automatic additivity for injective Jordan
	semi-triple maps from SMRs to their ambient matrix rings. The first obstruction
	comes from singleton connected components of $\Gamma_{\preceq}$, which give
	scalar direct summands.
	
	\begin{proposition}\label{prop:scalar-classification}
		For a division ring $\DD$, the following conditions are equivalent.
		\begin{enumerate}[label=\textup{(\roman*)}]
			\item Every injective scalar Jordan semi-triple map $f:\DD\to\DD$ is additive.
			\item Every injective scalar Jordan semi-triple map has the form $f(x)=\eps\theta(x)$, where $\eps\in\{1,-1\}$ and $\theta$ is an injective endomorphism or an injective anti-endomorphism of $\DD$.
			\item $\DD$ is isomorphic to one of $\FF_2,\FF_3,\FF_4$.
		\end{enumerate}
	\end{proposition}
	
	\begin{proof}
		The implication (ii) $\implies$ (i) follows because endomorphisms and anti-endomorphisms are additive by definition, and multiplication by the fixed sign $\eps=\pm1$ preserves additivity.
		
		Now we prove (i) $\implies$ (iii). Define a map
		\[
		\eta : \DD \to \DD, \qquad \eta(a):=\begin{cases} a^{-1},  &\text{ if } a \neq 0,\\
			0,  &\text{ if } a=0.
		\end{cases}
		\]
		It is easy to check that $\eta$ is an injective scalar Jordan semi-triple map. By assumption, $\eta$ is additive. For $a\notin\{0,-1\}$,
		\[
		(a+1)^{-1}=\eta(a+1)=\eta(a)+\eta(1)=a^{-1}+1.
		\]
		Multiplying on the left by $a$ and on the right by $a+1$ gives
		\[
		a=a(a^{-1}+1)(a+1)=(1+a)(a+1)=1+2a+a^2,
		\]
		so $a^2+a+1=0$. Thus $a^3=1$ for all $a\notin\{0,-1\}$. Consequently every $x\in\DD$ satisfies $x^7=x$. By Jacobson's commutativity theorem, a division ring in which every element is periodic is commutative; see \cite[Theorem~11]{Jacobson-potent} or the modern exposition \cite[Theorem~4.81]{Bresar-NCA}. We may therefore argue
		inside a field. Since $t^2+t+1$ has at most two roots, every element of $\DD$ belongs to the union of $\{0,-1\}$ and the set of roots of $t^2+t+1$. Hence $|\DD|\le4$, and $\DD\cong\FF_2,\FF_3$, or $\FF_4$.
		
		It remains to show (iii) $\implies$ (ii). Let $\DD$ be one of these fields, and let $f:\DD\to\DD$ be an injective scalar Jordan semi-triple map. First $f(0)=0$: otherwise $e:=f(0)\ne0$ would satisfy $e=e f(y)e$ for all $y$; multiplying by $e^{-1}$ on the left and on the right gives $f(y)=e^{-1}$ for all $y$, contradicting injectivity. Over $\FF_2$, injectivity gives $f=\operatorname{id}$. Over $\FF_3$, since $f(1)\ne0$, we have $f(1)=\pm1$, and injectivity gives $f=\pm\operatorname{id}$. Over $\FF_4$, the relation $f(1)^3=f(1)$ gives $f(1)^2=1$, so $f(1)=1$; hence $f$ fixes $0$ and $1$. Let $\omega\in\FF_4\setminus\FF_2$. Since $\FF_4^\times=\{1,\omega,\omega^2\}$ is cyclic of order three and $\omega^2+\omega+1=0$, the remaining two elements are either fixed or interchanged. These two possibilities are respectively the identity and the Frobenius automorphism $x\mapsto x^2$. Both are field automorphisms and hence satisfy the scalar Jordan semi-triple identity. These are precisely the maps in (ii).
	\end{proof}
	
	\begin{remark}\label{rem:additive-scalar-JST}
		Once additivity is imposed, scalar Jordan semi-triple maps reduce to the
		expected multiplicative or anti-multiplicative form. Let $f:\DD\to\DD$ be a nonzero additive scalar Jordan semi-triple map. Then $f$ is
		injective. Indeed, if $f(a)=0$ for some $a\ne0$, then for every $x\in\DD$,
		\[
		f(x) = f\left(a(a^{-1}xa^{-1})a\right)  =  f(a)f(a^{-1}xa^{-1})f(a)  =  0,
		\]
		contrary to $f\ne0$. Thus $\ker f=\{0\}$, and additivity gives injectivity. Set $\eps:=f(1)$. Then $\eps^3=\eps$. Since $f$ is injective and $f(0)=0$, we have $\eps\ne0$. Hence $\eps^2=1$, and therefore $\eps=\pm1$. Set
		\[
		\theta:\DD\to\DD,\qquad \theta:=\eps f .
		\]
		The map $\theta$ is again an injective additive scalar Jordan
		semi-triple map with $\theta(1)=1$. In particular,
		\[
		\theta(x^2)=\theta(x)^2,\qquad \text{for all } x\in\DD.
		\]
		Thus $\theta$ is a unital Jordan endomorphism of $\DD$. Since the codomain is a division ring and hence has no zero divisors, the Jacobson--Rickart theorem \cite[Theorem~2]{JR} applies. Hence $\theta$ is either an injective
		endomorphism or an injective anti-endomorphism of $\DD$. Therefore
		\[
		f=\eps\theta,\qquad \eps\in\{1,-1\}.
		\]
	\end{remark}
	
	The second obstruction is order-theoretic, and we formalize it in the following definition.
	
	\begin{definition}\label{def:nonisolated-pair}
		Let $\preceq$ be a preorder on $[n]$. Let $(i,j)$ be a nonsymmetric comparable pair, i.e.\ $i\preceq j$ and $j\not\preceq i$. We say that $(i,j)$ is \emph{nonisolated} if there exists $k\in[n]\setminus\{i,j\}$ such that at least one of
		\[
		k\preceq i\preceq j,
		\qquad
		i\preceq k\preceq j,
		\qquad
		i\preceq j\preceq k
		\]
		holds. We refer to these cases as \emph{lower-endpoint}, \emph{middle}, and \emph{upper-endpoint nonisolation}, respectively. A nonsymmetric comparable pair which fails to be nonisolated is called \emph{isolated}.
	\end{definition}
	
	\begin{example}\label{ex:T_n-SMR}
		The upper-triangular ring $T_n(\DD)$ is the SMR attached to the usual total order
		$1\leq 2\leq \cdots\leq n$. If $n\ge3$, then every (nonsymmetric)
		comparable pair in $T_n(\DD)$ is nonisolated. Indeed, let $i<j$. If there is an
		index $k$ with $i<k<j$, then $k$ gives middle nonisolation. If no such $k$
		exists, then $j=i+1$; since $n\ge3$, either $i>1$ or $j<n$, and a neighbouring
		index gives lower-endpoint or upper-endpoint nonisolation. In contrast, in
		$T_2(\DD)$ the only nonsymmetric comparable pair is $(1,2)$, and it is
		isolated.
		
		The same observation applies to every intermediate subring
		$T_n(\DD)\subseteq\R\subseteq M_n(\DD)$. By
		Remark~\ref{rem:diagonal-characterization}, such an $\R$ is the SMR attached
		to the preorder $\preceq_{\R}$ defined by
		\[
		i\preceq_{\R}j
		\quad\Longleftrightarrow\quad
		\DD E_{ij}\subseteq\R .
		\]
		Since $T_n(\DD)\subseteq\R$, the preorder $\preceq_{\R}$ contains the usual
		total order $\leq$. Hence the equivalence classes of its symmetric part $\asymp_{\R}$ are
		consecutive intervals: if $i<k<j$ and $i\asymp_{\R}j$, then
		$i\preceq_{\R}k\preceq_{\R}j\preceq_{\R}i$, and so $k\asymp_{\R}i$. Thus
		there is a partition
		\[
		[n]=B_1\sqcup\cdots\sqcup B_m
		\]
		into nonempty consecutive blocks such that
		\[
		i\preceq_{\R}j
		\quad\Longleftrightarrow\quad
		i\in B_p,\ j\in B_q,\ \text{and }p\le q .
		\]
		Thus $\R$ is a block upper-triangular ring with full diagonal blocks
		$B_1,\ldots,B_m$. Conversely, every such block upper-triangular ring clearly
		lies between $T_n(\DD)$ and $M_n(\DD)$. 
		
		Assume now that $i\prec_{\R}j$ and $j\not\preceq_{\R}i$. Then
		$i<j$ in the usual order. If $n\ge3$, the argument from the first paragraph
		gives a third index below $i$, between $i$ and $j$, or above $j$, and hence
		the pair is nonisolated. As already noted, the only exceptional case is
		$T_2(\DD)$.
		
		For related work on block upper-triangular rings and algebras and on associated preserver problems, see, for example, \cite{GPT-block-upper-triangular} and the references therein.
	\end{example}
	
	We have the following straightforward lemma. 
	\begin{lemma}\label{lem:isolated-equiv}
		Let $\preceq$ be a preorder on $[n]$. Assume that $i\preceq j$ and $j\not\preceq i$. Then $(i,j)$ is isolated if and only if, for every $h\in[n]$,
		\begin{equation}\label{eq:isolated-equiv}
			h\preceq i\implies h=i,
			\qquad
			i\preceq h\preceq j\implies h\in\{i,j\},
			\qquad
			j\preceq h\implies h=j.
		\end{equation}
		Consequently, if $(i,j)$ is isolated, then no vertex outside $\{i,j\}$ is equivalent to either endpoint.
	\end{lemma}
	
	\begin{proof}
		If the first implication in \eqref{eq:isolated-equiv} fails, then there is $h\ne i$ with $h\preceq i$. Since $j\not\preceq i$, this $h$ cannot be $j$, and $h\preceq i\preceq j$ gives lower-endpoint nonisolation. If the second implication fails, some $h\notin\{i,j\}$ satisfies $i\preceq h\preceq j$, giving middle nonisolation. If the third implication fails, then there is $h\ne j$ with $j\preceq h$. Since $j\not\preceq i$, this $h$ cannot be $i$, and $i\preceq j\preceq h$ gives upper-endpoint nonisolation. Conversely, any nonisolated configuration is precisely a failure of the corresponding implication. The final assertion follows from the first and third implications.
	\end{proof}
	
	\begin{proposition}\label{prop:bad-interval}
		Let $\R_{\preceq} \subseteq M_n(\DD)$ be an SMR with $|Z(\DD)|>2$. Assume that $i \preceq j$ is a nonsymmetric isolated pair. Then there exists a nonadditive injective Jordan semi-triple map $\Theta: \R_{\preceq} \to \R_{\preceq}$.
	\end{proposition}
	
	\begin{proof}
		Choose $\lambda\in Z(\DD)\setminus\{0,1\}$. Define $\Theta:\R_{\preceq}\to\R_{\preceq}$ by fixing all entries except the $(i,j)$-entry and setting
		\[
		(\Theta(X))_{ij}:=\begin{cases}
			\lambda x_{ij},&x_{ii}=0\text{ and }x_{jj}=0,\\
			x_{ij},&\text{otherwise.}
		\end{cases}
		\]
		The map is clearly injective. It is not additive, since
		\[
		\Theta(E_{ij})=\lambda E_{ij},
		\qquad
		\Theta(E_{ii}+E_{ij})=E_{ii}+E_{ij}
		\ne E_{ii}+\lambda E_{ij}=\Theta(E_{ii})+\Theta(E_{ij}).
		\]
		We now prove that $\Theta$ satisfies the Jordan semi-triple identity. By
		Lemma~\ref{lem:isolated-equiv}, the absence of vertices below $i$ and
		above $j$ implies that for every $X=[x_{pq}]\in\R_{\preceq}$ and $a \in \DD$
		we have
		\begin{equation}\label{eq:bad-E-identities}
			X(aE_{ij})=(x_{ii}a)E_{ij},\qquad
			E_{ij}(aX)=(ax_{jj})E_{ij}.    
		\end{equation}
		Fix $X=[x_{pq}], Y=[y_{pq}] \in \R_{\preceq}$. Again, by invoking Lemma~\ref{lem:isolated-equiv}, we obtain 
		\begin{align}\label{eq:bad-interval-entry}
			(XYX)_{ij}
			&=
			x_{ii}y_{ii}x_{ij}
			+
			x_{ii}y_{ij}x_{jj}
			+
			x_{ij}y_{jj}x_{jj},\\
			(XYX)_{ii}& =x_{ii}y_{ii}x_{ii}, \qquad (XYX)_{jj}=x_{jj}y_{jj}x_{jj}. \nonumber  
		\end{align}
		For $Z=[z_{pq}]\in\R_{\preceq}$, write
		\[
		\varepsilon_Z:=
		\begin{cases}
			\lambda, & z_{ii}=z_{jj}=0,\\
			1, & \text{otherwise}.
		\end{cases}
		\]
		Then
		\[
		\Theta(Z)=Z+(\varepsilon_Z-1)z_{ij}E_{ij}.
		\]
		Since $\lambda\in Z(\DD)$, the coefficient $\varepsilon_Z-1$ is central.
		Using \eqref{eq:bad-E-identities}, all terms in
		$\Theta(X)\Theta(Y)\Theta(X)$ containing at least two deformation terms $E_{ij}$ 
		vanish. Hence
		\begin{align}\label{eq:theta(X)theta(Y)theta(X)}
			\Theta(X)\Theta(Y)\Theta(X)&= XYX
			+XY\left((\varepsilon_X-1)x_{ij}E_{ij}\right) 
			+X\left((\varepsilon_Y-1)y_{ij}E_{ij}\right)X  \nonumber \\
			& \hspace{1.5cm} +\left((\varepsilon_X-1)x_{ij}E_{ij}\right)YX \\ \nonumber
			&=
			XYX
			+(\varepsilon_X-1)x_{ii}y_{ii}x_{ij}E_{ij}
			+(\varepsilon_Y-1)x_{ii}y_{ij}x_{jj}E_{ij} 
			+(\varepsilon_X-1)x_{ij}y_{jj}x_{jj}E_{ij}. \nonumber
		\end{align}
		Therefore, by \eqref{eq:bad-interval-entry}, the $(i,j)$-entry of
		$\Theta(X)\Theta(Y)\Theta(X)$ is
		\begin{equation}\label{eq:bad-rhs-entry}
			\varepsilon_X
			\left(x_{ii}y_{ii}x_{ij}+x_{ij}y_{jj}x_{jj}\right)
			+
			\varepsilon_Y x_{ii}y_{ij}x_{jj}.
		\end{equation}
		On the other hand, the $(i,j)$-entry of $\Theta(XYX)$ is obtained from
		\eqref{eq:bad-interval-entry} by multiplying it by $\lambda$ precisely when
		\[
		x_{ii}y_{ii}x_{ii}=0
		\quad\text{and}\quad
		x_{jj}y_{jj}x_{jj}=0.
		\]
		We compare the two expressions. If $x_{ii}=x_{jj}=0$, then
		\eqref{eq:bad-interval-entry} is zero, so both $(i,j)$-entries are zero.
		If $y_{ii}=y_{jj}=0$, then \eqref{eq:bad-interval-entry} reduces to
		$x_{ii}y_{ij}x_{jj}$, and both sides multiply this term by $\lambda$. It remains to consider the case where neither
		$(x_{ii},x_{jj})$ nor $(y_{ii},y_{jj})$ is $(0,0)$. Then
		$\varepsilon_X=\varepsilon_Y=1$, so \eqref{eq:bad-rhs-entry} equals
		\eqref{eq:bad-interval-entry}. If at least one of
		$x_{ii}y_{ii}x_{ii}$ and $x_{jj}y_{jj}x_{jj}$ is nonzero, then
		$\Theta$ also does not deform $XYX$, and the two entries agree. If both
		are zero, then, since $\DD$ has no zero divisors,
		\[
		x_{ii}=0 \, \text{ or } \, y_{ii}=0,
		\qquad
		x_{jj}=0 \, \text{ or } \, y_{jj}=0.
		\]
		After excluding the two already treated possibilities, the only remaining
		cases are
		\[
		x_{ii}=y_{jj}=0,
		\qquad\text{or}\qquad
		y_{ii}=x_{jj}=0.
		\]
		In both cases \eqref{eq:bad-interval-entry} is zero. Hence the
		$(i,j)$-entries of $\Theta(XYX)$ and
		$\Theta(X)\Theta(Y)\Theta(X)$ agree in all cases.  All other entries agree automatically. Indeed, by the definition of $\Theta$, we have \[
		\supp(\Theta(Z)-Z) \subseteq \{(i,j)\}, \qquad \text{for all } Z\in \R_{\preceq}.
		\]
		Furthermore, the expansion \eqref{eq:theta(X)theta(Y)theta(X)} shows that 
		\[
		\supp(\Theta(X)\Theta(Y)\Theta(X)-XYX) \subseteq \{(i,j)\}.
		\]
		Therefore, the matrices $\Theta(X)\Theta(Y)\Theta(X)$ and $\Theta(XYX)$ possibly differ only in the position $(i,j)$. As we already checked that their $(i,j)$-elements agree,  it follows that $\Theta(X)\Theta(Y)\Theta(X)=\Theta(XYX)$. Therefore, $\Theta$ is a Jordan semi-triple map, as claimed.
	\end{proof}
	
	\begin{remark}\label{rem:comparison-ps}
		The nonisolation condition is the analogue, in the present setting, of the
		local graph condition used in the study of injective commutativity- and
		spectrum-preserving maps on structural matrix algebras; see
		\cite[Theorem~3.7 and Remark~3.8]{GT-Petek-Semrl}. There, a comparable
		nondiagonal coordinate is controlled by the existence of a third vertex
		comparable with both endpoints. For a nonsymmetric comparable pair
		$i\preceq j$, $j\not\preceq i$, this common-neighbour condition is exactly
		our trichotomy: the third vertex lies below $i$, between $i$ and $j$, or above
		$j$.
	\end{remark}
	
	We now state the main result of the paper:
	
	\begin{theorem}\label{thm:main}
		Let $\DD$ be a division ring with $|Z(\DD)|>2$, and let $\R_{\preceq}\subseteq M_n(\DD)$, $n \ge 1$, be an SMR associated with the preorder $\preceq$ on $[n]$. Then every injective Jordan semi-triple map $\phi:\R_{\preceq}\to M_n(\DD)$ is additive if and only if the following two conditions hold:
		\begin{description}[leftmargin=2.7em,style=nextline]
			\item[$(\mathsf P)$] every nonsymmetric comparable pair $i\preceq j$, $j\not\preceq i$, is nonisolated;
			\item[$(\mathsf S)$] if $\Gamma_{\preceq}$ has a singleton connected component, then $\DD\cong\FF_3$ or $\DD\cong\FF_4$.
		\end{description}
	\end{theorem}
	
	The proof of Theorem~\ref{thm:main} is given in the next section. We first
	point out how the centre-size hypothesis is used. The sufficiency part only
	requires $|\DD|>2$: if $(\mathsf P)$ and $(\mathsf S)$ hold, then every
	injective Jordan semi-triple map $\R_{\preceq}\to M_n(\DD)$ is additive. The
	stronger assumption $|Z(\DD)|>2$ is needed for the converse, specifically in
	Proposition~\ref{prop:bad-interval}, where the deformation of an isolated
	nonsymmetric comparable pair requires a central scalar different from $0$ and
	$1$. The examples in Section~\ref{sec:centre-size-hypothesis} show that both
	restrictions are genuine: if $\DD\cong\FF_2$, then $(\mathsf P)$ and
	$(\mathsf S)$ need not imply automatic additivity, while if
	$Z(\DD)\cong\FF_2$, automatic additivity may hold even when $(\mathsf P)$
	fails.
	
	\begin{remark}\label{rem:triangular-applications}
		By Example~\ref{ex:T_n-SMR}, condition $(\mathsf P)$ holds for $T_n(\DD)$
		whenever $n\ge3$, and more generally for any intermediate ring $T_n(\DD) \subseteq \R \subseteq M_n(\DD)$ (i.e.\ a block upper-triangular ring). In all these cases the comparability
		graph is connected and nonsingleton, so condition $(\mathsf S)$ is void.
		Hence, if $|Z(\DD)|>2$, Theorem~\ref{thm:main} implies that every injective
		Jordan semi-triple map from any such ring $\R$ into $M_n(\DD)$ is additive.
	\end{remark}
	
	\section{Proof of the main theorem}\label{sec:proof-main}
	
	Throughout this section fix a preorder $\preceq$ on $[n]$, and write $\R:=\R_{\preceq}$ for the corresponding SMR. Let $\DD$ be a fixed division ring with $\abs{\DD} > 2$. The proof proceeds in several steps, and we do not assume $(\mathsf P)$ or $(\mathsf S)$ until the final step. If $n=1$, then
	$\R\cong\DD$, condition $(\mathsf P)$ is vacuous, and the assertion follows
	from Proposition~\ref{prop:scalar-classification}. We shall therefore assume
	throughout that $n\ge2$.
	
	\subsection{Normalization and diagonal idempotents}\label{sec:normalization}
	The following normalization argument is inspired by \cite[Proposition~3 and Corollary~4]{Lesnjak-Sze}. We include the details because we deal with a structural matrix ring with coefficients in an arbitrary division ring.
	
	\begin{lemma}\label{lem:normalization}
		Let $\phi:\R\to M_n(\DD)$ be an injective Jordan semi-triple map. Further, set
		$J:=\phi(I)$ and define
		\begin{equation}\label{eq:definition-of-Phi}
			\Phi:\R\to M_n(\DD), \qquad \Phi(X):=J\phi(X).
		\end{equation}
		Then $\Phi$ is an injective Jordan semi-triple map and the following hold
		\[
		\Phi(0)=0,
		\qquad
		\Phi(I)=I,
		\qquad
		J^2=I,
		\qquad
		\phi(0)=0.
		\]
		Moreover, $J$ commutes with $\Phi(\R)$, $\Phi$ maps idempotents in $\R$ to
		idempotents in $M_n(\DD)$, and
		\[
		\rank\Phi(P_S)=|S|,
		\qquad S\subseteq[n].
		\]
		Finally, $\phi$ is additive if and only if $\Phi$ is additive.
	\end{lemma}
	
	\begin{proof}
		Applying the Jordan semi-triple identity with $X=Y=I$ gives $J^3=J$. Let $X \in \R$. Since $X=IXI$, it follows that
		\begin{equation}\label{eq:phi-J-sandwich}
			\phi(X)=J\phi(X)J.
		\end{equation}
		Then 
		\[
		J^2\phi(X)J^2
		=
		J^2(J\phi(X)J)J^2
		=
		J^3\phi(X)J^3
		=
		\phi(X).
		\]
		Since $J^2$ is idempotent, multiplying this equality by $J^2$ on the left and on the right gives
		\begin{equation}\label{eq:J^2phi}
			J^2\phi(X)=\phi(X)=\phi(X)J^2 .
		\end{equation}
		Hence, 
		\[
		J\phi(X)
		\stackrel{\eqref{eq:phi-J-sandwich}}=
		J(J\phi(X)J)
		=
		J^2\phi(X)J
		\stackrel{\eqref{eq:J^2phi}}=
		\phi(X)J .
		\]
		Thus $J$ commutes with the image $\phi(\R)$ and with $\Phi(\R)=J\phi(\R)$.
		Consequently, for $X,Y \in \R$,
		\begin{align*}
			\Phi(X)\Phi(Y)\Phi(X)
			&=
			(J\phi(X))(J\phi(Y))(J\phi(X))  =
			J^3\phi(X)\phi(Y)\phi(X)
			=
			J\phi(XYX)\\
			&=
			\Phi(XYX),
		\end{align*}
		so $\Phi$ is a Jordan semi-triple map. If $\Phi(X)=\Phi(Y)$, then
		\[
		\phi(X)
		\stackrel{\eqref{eq:J^2phi}}=
		J^2\phi(X)
		=
		J\Phi(X)
		=
		J\Phi(Y)
		=
		J^2\phi(Y)
		\stackrel{\eqref{eq:J^2phi}}=
		\phi(Y),
		\]
		and injectivity of $\phi$ gives $X=Y$.
		
		Let $P\in\R$ be idempotent. Since $P=PIP$, we have
		\[
		\phi(P)=\phi(P)J\phi(P).
		\]
		Multiplying by $J$ on the left and using 
		$J\phi(P) = \phi(P)J$ gives
		\[
		\Phi(P)= J\phi(P) = J(\phi(P)J\phi(P)) = J^2 \phi(P)^2 \stackrel{\eqref{eq:J^2phi}}= \phi(P)^2.
		\]
		Since $P=P^3$, also $\phi(P)^3=\phi(P)$, and therefore
		\[
		\Phi(P)^2=\phi(P)^4=\phi(P)^2=\Phi(P).
		\]
		Thus $\Phi$ maps idempotents in $\R$ to idempotents in $M_n(\DD)$.
		
		For arbitrary idempotents $P,Q \in M_n(\DD)$, write
		\[
		P\trianglelefteq_s Q
		\quad\stackrel{\text{def}}\Longleftrightarrow\quad
		QPQ=P,
		\]
		and write $P\vartriangleleft_s Q$ if $P\trianglelefteq_s Q$ and $P\ne Q$.
		Note that
		\begin{equation}\label{eq:rank-observation}
			P\vartriangleleft_s Q \quad\implies\quad \rank P<\rank Q.
		\end{equation}
		Indeed, $QPQ=P$ implies $\operatorname{Im}P\subseteq\operatorname{Im}Q$.
		If these images had the same right $\DD$-dimension, then they would be
		equal. Since an idempotent acts as the identity on its image, this would give
		$QP=P$ and $PQ=Q$, whence $QPQ=Q$, contrary to $P\ne Q$.
		
		The Jordan semi-triple identity implies directly that
		\[
		P \trianglelefteq_s Q \quad \implies \quad \Phi(P) \trianglelefteq_s \Phi(Q).
		\]
		The diagonal idempotents
		\[
		P_0:=0,
		\qquad
		P_r:=E_{11}+\cdots+E_{rr}, \quad 1\le r\le n,
		\]
		form a strict sandwich chain
		\[
		0=P_0\vartriangleleft_s P_1\vartriangleleft_s\cdots\vartriangleleft_s P_n=I.
		\]
		By the injectivity of $\Phi$, their images under $\Phi$ form a strict sandwich chain of idempotents in
		$M_n(\DD)$. Hence \eqref{eq:rank-observation} gives
		\[
		\rank\Phi(P_0)<\rank\Phi(P_1)<\cdots<\rank\Phi(P_n).
		\]
		Since each rank is an integer between $0$ and $n$, this sequence must be $0<1<\cdots<n$. In particular,
		\[
		\Phi(0)=0,
		\qquad
		\Phi(I)=I,
		\]
		which also directly implies $\phi(0) = 0$ and $J^2 = \phi(I)^2 = I$. Further, since $J$ is invertible,  it is clear that $\phi$ is additive if and
		only if $\Phi$ is additive.
		
		We now prove the rank assertion for all diagonal idempotents. Let
		$S\subseteq[n]$ and put $s:=|S|$. Choose enumerations
		\[
		S=\{i_1,\ldots,i_s\},
		\qquad
		[n]\setminus S=\{j_1,\ldots,j_{n-s}\}.
		\]
		Then
		\[
		0
		\vartriangleleft_s
		E_{i_1i_1}
		\vartriangleleft_s
		(E_{i_1i_1}+E_{i_2i_2})
		\vartriangleleft_s
		\cdots
		\vartriangleleft_s
		P_S
		\]
		is a strict sandwich chain of length $s+1$ from $0$ to $P_S$, and
		\[
		P_S
		\vartriangleleft_s
		(P_S+E_{j_1j_1})
		\vartriangleleft_s
		(P_S+E_{j_1j_1}+E_{j_2j_2})
		\vartriangleleft_s
		\cdots
		\vartriangleleft_s
		I
		\]
		is a strict sandwich chain of length $n-s+1$ from $P_S$ to $I$. Applying \eqref{eq:rank-observation} to the images of these two chains gives $\rank\Phi(P_S)\ge s$ and $\rank\Phi(P_S)\le s$, respectively. Therefore, $\rank\Phi(P_S)=|S|$.
	\end{proof}
	
	Given an injective Jordan semi-triple map $\phi: \R \to M_n(\DD)$, from now until Proposition~\ref{prop:suff-additivity} we work with a normalized map from \eqref{eq:definition-of-Phi}, denoted by $\Phi$. Furthermore, replacing $\Phi$ by a map $X\mapsto S\Phi(X)S^{-1}$, where $S \in \GL_n(\DD)$, preserves all the properties stated in Lemma~\ref{lem:normalization} and also preserves additivity. We shall use this convention without further comment.
	
	The following lemma is inspired by the elementary matrix algebra argument in the normalization method of \cite[Section~2]{Lesnjak-Sze}.
	
	\begin{lemma}\label{lem:two-point-rigidity}
		For $i \ne j$, we have
		\[
		\Phi(E_{ii})\Phi(E_{jj}) = \Phi(E_{jj})\Phi(E_{ii}) = 0, \qquad \Phi(E_{ii}+ E_{jj}) = \Phi(E_{ii})+\Phi(E_{jj}).
		\]
	\end{lemma}
	
	\begin{proof}
		Denote
		\[
		P:=\Phi(E_{ii}),\qquad
		R:=\Phi(E_{jj}),\qquad
		Q:=\Phi(E_{ii}+E_{jj}).
		\]
		By Lemma~\ref{lem:normalization}, $P$ and $R$ are rank-one idempotents, while
		$Q$ is an idempotent of rank two. The identities
		\[
		E_{rr}=(E_{ii}+E_{jj})E_{rr}(E_{ii}+E_{jj}), \qquad r\in\{i,j\}
		\]
		give $P=QPQ$ and $R=QRQ$. Thus $P$ and $R$ belong to the corner $QM_n(\DD)Q$. As $\operatorname{Im} Q$ is a two-dimensional right $\DD$-space and $Q$ acts as the identity on it, the restriction to $\operatorname{Im} Q$ identifies the corner $QM_n(\DD)Q$ with the endomorphism ring $\operatorname{End}_{\DD}^{\,r}(\operatorname{Im} Q)$ of the right $\DD$-vector space $\operatorname{Im} Q$. After choosing a right $\DD$-basis of $\operatorname{Im} Q$, we identify this corner with $M_2(\DD)$. After composing the restricted map $\Phi|_{\DD E_{ii}\oplus \DD E_{jj}}$
		with this corner identification, we may work inside $M_2(\DD)$ and regard
		$Q$ as $I_2$. Since $P$ is a rank-one idempotent in $M_2(\DD)$, a further similarity inside
		$M_2(\DD)$ allows us to assume that $P=E_{11}$. The relation $E_{ii}E_{jj}E_{ii}=0$ then gives $PRP=0$.
		
		We first determine the possible form of $R$. Since $R_{11} = 0$, we can write
		\[
		R=\begin{bmatrix}0&\alpha\\ \beta&u\end{bmatrix},\qquad \alpha,\beta,u \in \DD.
		\]
		The equation $R^2=R$  gives
		\[
		\alpha\beta=0,
		\qquad \alpha u=\alpha,
		\qquad u\beta=\beta,
		\qquad \beta\alpha+u^2=u.
		\]
		Hence at least one of $\alpha,\beta$ is zero. If both are zero, $\rank(R) = 1$ forces $u=1$. If exactly one is nonzero, then the cancellation from the second or third equality again gives $u=1$. Thus we obtain one of the following:
		\[
		R=E_{22},
		\qquad R=E_{22}+\alpha E_{12}\ (\alpha\in \DD^\times),
		\qquad
		R=E_{22}+\beta E_{21}\ (\beta\in \DD^\times).
		\]
		A diagonal similarity inside $M_2(\DD)$, which fixes $P=E_{11}$ and preserves all hypotheses, reduces the two nonorthogonal cases respectively to
		\[
		R=E_{22}+E_{12}
		\qquad\text{or}\qquad
		R=E_{22}+E_{21}.
		\]
		Indeed, conjugation by $\diag(1,\alpha)$ sends $E_{22}+\alpha E_{12}$ to $E_{22}+E_{12}$, and conjugation by $\diag(1,\beta^{-1})$ sends $E_{22}+\beta E_{21}$ to $E_{22}+E_{21}$.
		
		We now rule out both nonorthogonal alternatives at once. Assume that
		$R=E_{22}+E_{12}$ or $R=E_{22}+E_{21}$. Put
		\[
		A_b:=\Phi(E_{ii}+bE_{jj})
		=
		\begin{bmatrix}
			r_b&s_b\\
			t_b&u_b
		\end{bmatrix}, \qquad r_b,s_b,t_b,u_b \in \DD.
		\]
		By applying $\Phi$ to
		\[
		E_{ii}(E_{ii}+bE_{jj})E_{ii}=E_{ii},
		\qquad
		(E_{ii}+bE_{jj})E_{ii}(E_{ii}+bE_{jj})=E_{ii},
		\]
		we get $r_b=1$ and $s_b=t_b=0$. Hence $A_b=\diag(1,u_b)$.
		
		Note that for a
		general matrix $M=[m_{ij}] \in M_2(\DD)$, one computes
		\[
		(E_{22}+E_{12})M(E_{22}+E_{12})
		=
		(m_{21}+m_{22})(E_{22}+E_{12}),
		\]
		\[
		(E_{22}+E_{21})M(E_{22}+E_{21})=(m_{12}+m_{22})(E_{22}+E_{21})
		\]
		and hence $RM_2(\DD)R = \DD R$. Since $E_{jj}(bE_{jj})E_{jj}=bE_{jj},$ we have
		\[
		\Phi(bE_{jj})=R\Phi(bE_{jj})R\in RM_2(\DD)R.
		\]
		Thus there is $\rho(b)\in\DD$ such that $\Phi(bE_{jj})=\rho(b)R.$
		On the other hand,
		\[
		E_{jj}(E_{ii}+bE_{jj})E_{jj}=bE_{jj},
		\]
		so $R A_b R=\Phi(bE_{jj})=\rho(b)R.$ Since $A_b=\diag(1,u_b)$, this gives $u_b=\rho(b)$. Finally,
		\[
		(E_{ii}+bE_{jj})E_{jj}(E_{ii}+bE_{jj})=b^2E_{jj}
		\]
		gives $A_bRA_b=\Phi(b^2E_{jj})=\rho(b^2)R$. Using $A_b=\diag(1,\rho(b))$, and considering separately
		$R=E_{22}+E_{12}$ and $R=E_{22}+E_{21}$, we get respectively
		\[
		\begin{bmatrix}
			0&\rho(b)\\
			0&\rho(b)^2
		\end{bmatrix}
		=
		\begin{bmatrix}
			0&\rho(b^2)\\
			0&\rho(b^2)
		\end{bmatrix},
		\qquad
		\begin{bmatrix}
			0&0\\
			\rho(b)&\rho(b)^2
		\end{bmatrix}
		=
		\begin{bmatrix}
			0&0\\
			\rho(b^2)&\rho(b^2)
		\end{bmatrix}.
		\]
		In both cases,
		\[
		\rho(b)=\rho(b^2)=\rho(b)^2.
		\]
		If $b\ne0$, then $\rho(b)\ne0$, since otherwise
		$\Phi(bE_{jj})=0=\Phi(0)$, contradicting injectivity. Hence $\rho(b)=1$ for
		every $b\ne0$. Since $|\DD|>2$, we can choose two distinct nonzero values of
		$b$ with the same image $\Phi(bE_{jj})=R$, again contradicting injectivity.
		
		It follows that the nonorthogonal alternatives are impossible, thus concluding $R=E_{22}$. Hence, 
		\[
		PR=RP=0,\qquad P+R=I_2=Q.
		\]
		Since all preceding reductions were implemented by range isomorphisms and
		similarities, these identities transfer back to the original corner, and the
		proof is complete.
	\end{proof}
	
	\begin{proposition}\label{prop:diagonal-straightening}
		There exists $T\in\GL_n(\DD)$ such that, after replacing $\Phi$ by
		$X\mapsto T\Phi(X)T^{-1}$, one has
		\[
		\Phi(P_S)=P_S, \qquad \text{for all } S\subseteq[n].
		\]
		Consequently
		\begin{equation}\label{eq:compression}
			\Phi(P_SXP_S)=P_S\Phi(X)P_S,
			\qquad \text{for all } X\in\R,
			\ S\subseteq[n].
		\end{equation}
	\end{proposition}
	
	\begin{proof}
		If $i \in [n]$, set $Q_i:=\Phi(E_{ii})$. For $i\ne j$, Lemma~\ref{lem:two-point-rigidity} gives
		\[
		Q_iQ_j=Q_jQ_i=0,
		\qquad
		\Phi(E_{ii}+E_{jj})=Q_i+Q_j.
		\]
		Hence, by Lemma~\ref{lem:normalization}, $Q_1,\ldots,Q_n$ are pairwise orthogonal rank-one idempotents. Therefore $Q_1+\cdots+Q_n$ is an idempotent of rank $n$, and hence equal to $I$.
		
		Choose a nonzero vector in each $\operatorname{Im}Q_i$. These vectors form a basis of the right $\DD$-vector space $\DD^n$, and in this basis the idempotents $Q_i$ become the standard diagonal idempotents $E_{ii}$. After the corresponding range similarity, we may assume
		\[
		\Phi(E_{ii})=E_{ii},\qquad \text{for all }i\in[n],
		\]
		and therefore $\Phi(E_{ii}+E_{jj})=E_{ii}+E_{jj}$ for $i\ne j$.
		
		Let $S\subseteq[n]$ and put $R:=\Phi(P_S)$. By
		Lemma~\ref{lem:normalization}, $R$ is an idempotent of rank $|S|$. For
		$i\in S$, the identity $P_SE_{ii}P_S=E_{ii}$ gives $RE_{ii}R=E_{ii}$.
		Therefore
		\[
		RP_SR
		=
		\sum_{i\in S}RE_{ii}R
		=
		\sum_{i\in S}E_{ii}
		=
		P_S.
		\]
		Thus, using the notation of the proof of Lemma~\ref{lem:normalization}, $P_S\trianglelefteq_s R$. Since $\rank P_S=|S|=\rank R$, the rank observation \eqref{eq:rank-observation} gives $R=P_S$. Formula \eqref{eq:compression} now follows directly from the Jordan semi-triple
		identity.
	\end{proof}
	
	From this point on, we assume that the suitable range similarity has been applied to $\Phi$, and therefore that \eqref{eq:compression} holds. We shall refer to it as the \emph{compression identity}. Following \cite{GT-matrix-jordan-mult,GT-mult}, it is convenient to reformulate it in terms of supports. 
	Namely, the compression identity says that support in a diagonal block
	is preserved by $\Phi$: if $X\in\R$ is supported in $S\times S$, where $S \subseteq [n]$, then $\Phi(X)$ is again supported in $S\times S$. We shall use this
	support-preservation form of \eqref{eq:compression} without further comment.
	
	\begin{lemma}\label{lem:two-point-compressions}
		Let $X\in\R$, let $i,j\in[n]$ be distinct, and write
		$P:=P_{\{i,j\}}$. Then
		\begin{align*}
			\Phi(X)_{ii} &= \Phi(PXP)_{ii}, &
			\Phi(X)_{ij} &= \Phi(PXP)_{ij}, \\
			\Phi(X)_{ji} &= \Phi(PXP)_{ji}, &
			\Phi(X)_{jj} &= \Phi(PXP)_{jj}.
		\end{align*}
		In particular, if $X$ is supported in $R\times R$ for some
		$R\subseteq[n]$ with $|R|\ge2$, then $\Phi(X)$ is recovered entrywise from
		the two-point compressed images
		\[
		\Phi(P_{\{i,j\}}XP_{\{i,j\}}),
		\qquad i,j\in R,\quad i\ne j.
		\]
	\end{lemma}
	
	\begin{proof}
		The compression identity \eqref{eq:compression} applied to $P=P_{\{i,j\}}$ yields      $\Phi(PXP)=P\Phi(X)P$. This equality directly implies
		\[
		\Phi(X)_{pq}=\Phi(PXP)_{pq},
		\qquad p,q\in\{i,j\}.
		\]
		Now suppose that $X$ is supported in $R\times R$, where $|R|\ge2$. By the
		support-preservation form of \eqref{eq:compression}, $\Phi(X)$ is also
		supported in $R\times R$. Every off-diagonal entry $\Phi(X)_{ij}$ with
		$i,j\in R$, $i\ne j$, occurs in the compression to $\{i,j\}$, and every
		diagonal entry $\Phi(X)_{ii}$ with $i\in R$ occurs in any compression to
		$\{i,j\}$ with $j\in R$, $j\ne i$. Since all entries outside $R\times R$
		are zero, these two-point compressed images recover $\Phi(X)$ entrywise.
	\end{proof}
	
	Although Lemma~\ref{lem:two-point-compressions} is stated for arbitrary $R$ with $|R|\ge2$, it will be
	used only for two- and three-point corners.
	
	\subsection{Local coordinates before additivity}\label{sec:local}
	
	In this subsection, we continue to assume that $\Phi$ is normalized and acts as the identity on all diagonal idempotents (Proposition~\ref{prop:diagonal-straightening}).
	
	\begin{lemma}\label{lem:coordinate-setup}
		For each $i\in [n]$ there exists an injective map $F_{ii}:\DD\to\DD$ with $F_{ii}(0)=0$ and $F_{ii}(1)=1$ such that
		\begin{equation}\label{eq:diag-coordinate}
			\Phi(aE_{ii})=F_{ii}(a)E_{ii}, \qquad a \in \DD.
		\end{equation}
		Moreover, for $i\ne j$,
		\begin{equation}\label{eq:two-diagonal-coordinate}
			\Phi(aE_{ii}+bE_{jj})=F_{ii}(a)E_{ii}+F_{jj}(b)E_{jj}, \qquad a,b \in \DD
		\end{equation}
		and hence
		\begin{equation}\label{eq:all-diagonal-coordinate}
			\Phi\left(\sum_{i=1}^n a_iE_{ii}\right)=\sum_{i=1}^nF_{ii}(a_i)E_{ii}, \qquad a_1, \ldots, a_n \in \DD.
		\end{equation}
		If $i\prec j$, then exactly one of the alternatives
		\begin{equation}\label{eq:pure-preserving}
			\Phi(aE_{ij})=F_{ij}(a)E_{ij}, \qquad a\in\DD,
		\end{equation}
		\begin{equation}\label{eq:pure-reversing}
			\Phi(aE_{ij})=F_{ij}(a)E_{ji}, \qquad a\in\DD,
		\end{equation}
		holds, where $F_{ij}:\DD\to\DD$ is injective and $F_{ij}(0)=0$. In the first case,
		\begin{equation}\label{eq:sandwich-preserving}
			F_{ij}(at)=F_{ii}(a)F_{ij}(t), 
			\qquad
			F_{ij}(ta)=F_{ij}(t)F_{jj}(a), \qquad a,t \in \DD,
		\end{equation}
		and in the second case,
		\begin{equation}\label{eq:sandwich-reversing}
			F_{ij}(at)=F_{ij}(t)F_{ii}(a),
			\qquad
			F_{ij}(ta)=F_{jj}(a)F_{ij}(t),  \qquad a,t \in \DD.
		\end{equation}
	\end{lemma}
	
	\begin{proof}
		Fix $i \in [n]$. For $a \in \DD$, the compression identity \eqref{eq:compression} for $S=\{i\}$ implies that $\Phi(aE_{ii})$ is supported in $\{i\} \times \{i\}$. Therefore, there exists a map $F_{ii}: \DD \to \DD$ such that \eqref{eq:diag-coordinate} holds. The injectivity of $F_{ii}$ follows from the injectivity of $\Phi$. Similarly, for $a,b\in\DD$ and $i\ne j$, applying \eqref{eq:compression} with $S=\{i,j\}$ gives
		\[
		\Phi(aE_{ii}+bE_{jj})=
		d_iE_{ii}+d_jE_{jj}+uE_{ij}+vE_{ji},
		\]
		for some $d_i,d_j,u,v \in \DD$.
		Comparing entries in
		\[
		\Phi(XYX)=\Phi(X)\Phi(Y)\Phi(X),
		\qquad
		\Phi(YXY)=\Phi(Y)\Phi(X)\Phi(Y),
		\]
		where $X=aE_{ii}+bE_{jj}$ and $Y \in \{E_{ii},E_{jj}\}$, gives $d_i = F_{ii}(a)$, $d_j = F_{jj}(b)$, and
		\[
		d_iu=0,
		\qquad vd_i=0,
		\qquad
		ud_j=0,
		\qquad d_jv=0.
		\]
		If $a\ne0$, then $d_i\ne0$ by the injectivity of $F_{ii}$ and $F_{ii}(0)=0$, and the first two equations force $u=v=0$; if $b\ne0$, the last two equations do the same. The case $a=b=0$ follows from $\Phi(0)=0$. This proves \eqref{eq:two-diagonal-coordinate}.
		
		To prove \eqref{eq:all-diagonal-coordinate}, let $D = \diag(a_1,\ldots,a_n) \in M_n(\DD)$ be an arbitrary diagonal matrix. For fixed $i \ne j$, applying \eqref{eq:compression} to the two-point set $\{i,j\}$ gives
		\[
		P_{\{i,j\}}DP_{\{i,j\}} = a_iE_{ii}+a_jE_{jj} \implies P_{\{i,j\}}\Phi(D)P_{\{i,j\}} \stackrel{\eqref{eq:two-diagonal-coordinate}}= F_{ii}(a_i)E_{ii} + F_{jj}(a_j)E_{jj}.
		\]
		Lemma~\ref{lem:two-point-compressions} yields
		\[
		\Phi(D)_{ii} = F_{ii}(a_i), \qquad \Phi(D)_{jj} = F_{jj}(a_j), \qquad \Phi(D)_{ij} = \Phi(D)_{ji} = 0.
		\]
		Since $i$ and $j$ were arbitrary, \eqref{eq:all-diagonal-coordinate} follows.
		
		Now fix $i\prec j$ and let $a \in \DD$. Compression to $\{i,j\}$ and the identities $E_{ii}(aE_{ij})E_{ii}=E_{jj}(aE_{ij})E_{jj}=0$ give
		\[
		\Phi(aE_{ij})=u(a)E_{ij}+v(a)E_{ji},
		\]
		for some $u(a),v(a) \in \DD$. Since $(aE_{ij})E_{ii}(aE_{ij})=0$, we get $v(a)u(a)=0$. For $a\ne0$, injectivity gives $\Phi(aE_{ij})\ne0$, so exactly one of $u(a),v(a)$ is nonzero; for $a=0$, both are zero.
		
		For $r,s\in\DD^\times$,
		\[
		(rE_{ii}+sE_{jj})(aE_{ij})(rE_{ii}+sE_{jj})=rasE_{ij}.
		\]
		The image of $rE_{ii}+sE_{jj}$ has two nonzero diagonal coefficients. Thus the orientation of $\Phi(aE_{ij})$ is preserved under replacing $a$ by $ras$, namely
		\[
		u(a) \ne 0 \iff u(ras) \ne 0 \qquad \text{and} \qquad v(a) \ne 0 \iff v(ras) \ne 0.
		\]
		Given $a,b\ne0$, choose $r:=1$ and $s:=a^{-1}b$. Hence the orientation is independent of $a\ne0$, and the corresponding nonzero coefficient function, denoted by \[
		F_{ij} : \DD^\times \to \DD^\times, \qquad F_{ij}(x) := \begin{cases}
			u(x), &\text{ if }u(x) \ne 0,\\
			v(x), &\text{ if }v(x) \ne 0,
		\end{cases}
		\] and extended by $F_{ij}(0):=0$, is injective. Finally, applying $\Phi$ to
		\[
		(aE_{ii}+E_{jj})(tE_{ij})(aE_{ii}+E_{jj})=atE_{ij},
		\qquad
		(E_{ii}+aE_{jj})(tE_{ij})(E_{ii}+aE_{jj})=taE_{ij}
		\]
		gives \eqref{eq:sandwich-preserving} and \eqref{eq:sandwich-reversing}, with the displayed coefficient order. The cases $a=0$ or $t=0$ are included because all coordinate maps involved send $0$ to $0$.
	\end{proof}
	
	For $i \preceq j$, when \eqref{eq:pure-preserving} holds we say that the pair $(i,j)$ is \emph{$\Phi$-preserving}; when \eqref{eq:pure-reversing} holds we say that the pair $(i,j)$ is \emph{$\Phi$-reversing}. For an off-diagonal pair $p\prec q$, write
	\begin{equation}\label{eq:Upq-definition}
		U_{pq}:=
		\begin{cases}
			E_{pq}, & \text{if }(p,q)\text{ is $\Phi$-preserving},\\
			E_{qp}, & \text{if }(p,q)\text{ is $\Phi$-reversing}.
		\end{cases}
	\end{equation}
	Thus \eqref{eq:pure-preserving} and \eqref{eq:pure-reversing} can be together stated as
	\begin{equation}\label{eq:pure-unified}
		\Phi(tE_{pq})=F_{pq}(t)U_{pq},\qquad t\in\DD.
	\end{equation}
	
	\begin{lemma}\label{lem:chain-orientation}
		Let $p,q,r,s\in[n]$ satisfy
		\[
		p\ne q,\qquad r\ne s,\qquad
		\{p,q\}\ne\{r,s\},\qquad
		|\{p,q,r,s\}|=3,
		\]
		and assume that $p\preceq q$ and $r\preceq s$. Then, for all
		$a,b\in\DD$,
		\begin{equation}\label{eq:pqrs-formula}
			\Phi(a E_{pq}+b E_{rs})
			=
			F_{pq}(a)U_{pq}+F_{rs}(b)U_{rs}.
		\end{equation}
		Moreover, if $i,j,k$ are distinct and $i\preceq j\preceq k$, then
		$(i,j)$, $(j,k)$, and $(i,k)$ have the same $\Phi$-orientation. In the
		$\Phi$-preserving case,
		\begin{equation}\label{eq:chain-coeff-preserving}
			F_{ik}(ab)=F_{ij}(a)F_{jk}(b),\qquad a,b \in \DD,
		\end{equation}
		and in the $\Phi$-reversing case,
		\begin{equation}\label{eq:chain-coeff-reversing}
			F_{ik}(ab)=F_{jk}(b)F_{ij}(a),\qquad a,b \in \DD.
		\end{equation}
	\end{lemma}
	
	\begin{proof}
		We first prove \eqref{eq:pqrs-formula}. For fixed $a, b \in \DD$, set
		\[
		R:=\{p,q,r,s\},
		\qquad
		X:=a E_{pq}+b E_{rs}.
		\]
		Since $X=P_RXP_R$, the support-preservation form of the compression identity \eqref{eq:compression}
		gives that $\Phi(X)$ is supported in $R\times R$. Also,
		\[
		P_{\{p,q\}}XP_{\{p,q\}}=a E_{pq},
		\qquad
		P_{\{r,s\}}XP_{\{r,s\}}=b E_{rs}.
		\]
		Hence, \eqref{eq:pure-unified} gives
		\[
		P_{\{p,q\}}\Phi(X)P_{\{p,q\}}
		=
		\Phi(a E_{pq})
		=
		F_{pq}(a)U_{pq},
		\]
		and similarly
		\[
		P_{\{r,s\}}\Phi(X)P_{\{r,s\}}
		=
		\Phi(b E_{rs})
		=
		F_{rs}(b)U_{rs}.
		\]
		For the remaining two-point subset $T\subset R$, distinct from
		$\{p,q\}$ and $\{r,s\}$, we have $P_TXP_T=0$, and therefore $P_T\Phi(X)P_T=\Phi(0)=0$. Since $\Phi(X)$ is supported in $R\times R$, by Lemma~\ref{lem:two-point-compressions}, these three two-point compressions determine all its entries. Thus
		\begin{equation*}
			\Phi(X)=F_{pq}(a)U_{pq}+F_{rs}(b)U_{rs},
		\end{equation*}
		which establishes \eqref{eq:pqrs-formula}. Now let $i,j,k$ be distinct with $i\preceq j\preceq k$. First take
		$a,b\in\DD^\times$, and set
		\[
		N:=aE_{ij}+bE_{jk}.
		\]
		By \eqref{eq:pqrs-formula},
		\[
		\Phi(N)=F_{ij}(a)U_{ij}+F_{jk}(b)U_{jk}.
		\]
		Since $NE_{jj}N=abE_{ik}$, we have
		\[
		\Phi(N)E_{jj}\Phi(N)=\Phi(abE_{ik}).
		\]
		The right-hand side is nonzero, because $ab\ne0$, $\Phi(0)=0$, and
		$\Phi$ is injective.
		
		If $(i,j)$ and $(j,k)$ had opposite $\Phi$-orientations, then the above
		formula would give $\Phi(N)E_{jj}\Phi(N)=0$, a contradiction. Hence $(i,j)$ and $(j,k)$ have the same $\Phi$-orientation. If both are $\Phi$-preserving, then
		\[
		\Phi(N)=F_{ij}(a)E_{ij}+F_{jk}(b)E_{jk},
		\]
		and consequently
		\[
		\Phi(abE_{ik})
		=
		\Phi(N)E_{jj}\Phi(N)
		=
		F_{ij}(a)F_{jk}(b)E_{ik}.
		\]
		Thus $(i,k)$ is $\Phi$-preserving and $F_{ik}(ab)=F_{ij}(a)F_{jk}(b)$ follows. If both are $\Phi$-reversing, then
		\[
		\Phi(N)=F_{ij}(a)E_{ji}+F_{jk}(b)E_{kj},
		\]
		and hence
		\[
		\Phi(abE_{ik})
		=
		\Phi(N)E_{jj}\Phi(N)
		=
		F_{jk}(b)F_{ij}(a)E_{ki}.
		\]
		Thus $(i,k)$ is $\Phi$-reversing and $F_{ik}(ab)=F_{jk}(b)F_{ij}(a)$ follows.
		
		This proves the orientation assertion and the coefficient identities \eqref{eq:chain-coeff-preserving} and \eqref{eq:chain-coeff-reversing} for
		$a,b\in\DD^\times$. If $a=0$ or $b=0$, then $ab=0$, and the identities follow from $F_{uv}(0)=0$.
	\end{proof}
	
	\begin{proposition}\label{prop:nonisolated-two-point}
		Let $i\ne j$. Suppose  that $i\preceq j$, $j\not\preceq i$, and $(i,j)$ is nonisolated.
		Let $a,b,\alpha\in\DD$ be arbitrary. Then $\Phi(aE_{ii}+bE_{jj}+\alpha E_{ij})$ equals  
		\[
		\begin{cases}
			F_{ii}(a)E_{ii}+F_{jj}(b)E_{jj}
			+F_{ij}(\alpha)E_{ij}, &\text{ if $(i,j)$ is $\Phi$-preserving},
			\\
			F_{ii}(a)E_{ii}+F_{jj}(b)E_{jj}
			+F_{ij}(\alpha)E_{ji}, &\text{ if $(i,j)$ is $\Phi$-reversing}.
		\end{cases}
		\]
	\end{proposition}
	
	\begin{proof}
		Let $a,b,\alpha\in\DD$.
		
		\textbf{Case 1.}
		Assume that the pair has an upper endpoint witness, say
		$i\preceq j\preceq k$ with $k\notin\{i,j\}$. By
		Lemma~\ref{lem:chain-orientation}, the coordinates $(i,j)$, $(j,k)$
		and $(i,k)$ have the same $\Phi$-orientation. Let
		\[
		A:=aE_{ii}+bE_{jj}+\alpha E_{ij}+E_{kk}.
		\]
		By \eqref{eq:compression} and Lemma~\ref{lem:coordinate-setup}, the $\{i,k\}$- and $\{j,k\}$-compressions
		of $\Phi(A)$ are
		\[
		P_{\{i,k\}}\Phi(A)P_{\{i,k\}}
		=
		\Phi(aE_{ii}+E_{kk})
		=
		F_{ii}(a)E_{ii}+E_{kk},
		\]
		and
		\[
		P_{\{j,k\}}\Phi(A)P_{\{j,k\}}
		=
		\Phi(bE_{jj}+E_{kk})
		=
		F_{jj}(b)E_{jj}+E_{kk}.
		\]
		By Lemma~\ref{lem:two-point-compressions}, the only possibly undetermined entries of $\Phi(A)$ in this
		three-point corner are the two off-diagonal entries in the $\{i,j\}$-corner. Hence
		\[
		\Phi(A)
		=
		F_{ii}(a)E_{ii}+F_{jj}(b)E_{jj}+E_{kk}
		+uE_{ij}+vE_{ji}
		\]
		for some $u,v\in\DD$. Note the identities
		\begin{equation}\label{eq:the-same-two-identities}
			AE_{jk}A = \alpha E_{ik}+bE_{jk}, \qquad AE_{ik}A = aE_{ik}.
		\end{equation}
		Suppose first that the common orientation is $\Phi$-preserving. By Lemma~\ref{lem:chain-orientation}, we have
		\begin{align*}
			uF_{jk}(1)E_{ik}
			+F_{jj}(b)F_{jk}(1)E_{jk} &= \Phi(A)\left(F_{jk}(1)E_{jk}\right)\Phi(A) = \Phi(\alpha E_{ik}+bE_{jk})
			\\
			&=
			F_{ik}(\alpha)E_{ik}+F_{jk}(b)E_{jk}.
		\end{align*}
		By \eqref{eq:sandwich-preserving}, we already know $F_{jj}(b)F_{jk}(1)=F_{jk}(b)$, so the comparison of $(i,k)$ coordinates gives
		$uF_{jk}(1)=F_{ik}(\alpha)$. Similarly,
		\[
		F_{ii}(a)F_{ik}(1)E_{ik}
		+vF_{ik}(1)E_{jk} = \Phi(A)\left(F_{ik}(1)E_{ik}\right)\Phi(A)=\Phi(aE_{ik})=
		F_{ik}(a)E_{ik},
		\]
		and $F_{ii}(a)F_{ik}(1)=F_{ik}(a)$ by
		\eqref{eq:sandwich-preserving}, so $vF_{ik}(1)=0$.
		Using $F_{ik}(\alpha)=F_{ij}(\alpha)F_{jk}(1)$ from
		Lemma~\ref{lem:chain-orientation}, and cancelling the nonzero elements
		$F_{jk}(1)$ and $F_{ik}(1)$, we obtain $u=F_{ij}(\alpha)$ and $v=0$.
		Thus
		\[
		\Phi(A)
		=
		F_{ii}(a)E_{ii}+F_{jj}(b)E_{jj}+E_{kk}
		+F_{ij}(\alpha)E_{ij}.
		\]
		If the common orientation is $\Phi$-reversing, the same two identities \eqref{eq:the-same-two-identities} instead give
		$F_{jk}(1)v=F_{ik}(\alpha)$ and $F_{ik}(1)u=0$, with the remaining
		coefficients accounted for by \eqref{eq:sandwich-reversing}. Since now
		$F_{ik}(\alpha)=F_{jk}(1)F_{ij}(\alpha)$, we get
		$v=F_{ij}(\alpha)$ and $u=0$. Hence
		\[
		\Phi(A)
		=
		F_{ii}(a)E_{ii}+F_{jj}(b)E_{jj}+E_{kk}
		+F_{ij}(\alpha)E_{ji}.
		\]
		Finally, compressing the last two displayed formulas to $\{i,j\}$ gives
		\[
		\Phi(aE_{ii}+bE_{jj}+\alpha E_{ij})
		=
		\begin{cases}
			F_{ii}(a)E_{ii}+F_{jj}(b)E_{jj}
			+F_{ij}(\alpha)E_{ij}, & \text{in the preserving case},\\
			F_{ii}(a)E_{ii}+F_{jj}(b)E_{jj}
			+F_{ij}(\alpha)E_{ji}, & \text{in the reversing case}.
		\end{cases}
		\]
		
		\textbf{Case 2.} The lower-endpoint case is similar, but the coefficient order is reversed.
		Suppose $k\notin\{i,j\}$ satisfies $k\preceq i\preceq j$. By
		Lemma~\ref{lem:chain-orientation}, the coordinates $(k,i)$, $(i,j)$
		and $(k,j)$ have the same $\Phi$-orientation. Set
		\[
		A:=E_{kk}+aE_{ii}+bE_{jj}+\alpha E_{ij}.
		\]
		The $\{k,i\}$- and $\{k,j\}$-compressions of $\Phi(A)$ are diagonal,
		so (by Lemma~\ref{lem:two-point-compressions}) the only unknown entries again occur in the $\{i,j\}$-corner. Hence
		\[
		\Phi(A)
		=
		E_{kk}+F_{ii}(a)E_{ii}+F_{jj}(b)E_{jj}
		+uE_{ij}+vE_{ji}
		\]
		for some $u,v\in\DD$. Consider the identities
		\[
		AE_{ki}A=aE_{ki}+\alpha E_{kj},
		\qquad
		AE_{kj}A=bE_{kj}.
		\]
		In the preserving case, applying $\Phi$ to the identities and using Lemma~\ref{lem:chain-orientation} gives
		\[
		F_{ki}(1)u=F_{kj}(\alpha),
		\qquad
		F_{kj}(1)v=0,
		\]
		with the remaining coefficients matching by \eqref{eq:sandwich-preserving}.
		Lemma~\ref{lem:chain-orientation} for $k\preceq i\preceq j$ gives
		$F_{kj}(\alpha)=F_{ki}(1)F_{ij}(\alpha)$; hence
		$u=F_{ij}(\alpha)$ and $v=0$. In the reversing case, the same two identities give
		\[
		vF_{ki}(1)=F_{kj}(\alpha),
		\qquad
		uF_{kj}(1)=0,
		\]
		with the remaining coefficients matching by \eqref{eq:sandwich-reversing}.
		Since now $F_{kj}(\alpha)=F_{ij}(\alpha)F_{ki}(1)$, we get
		$v=F_{ij}(\alpha)$ and $u=0$. In either case, the final result follows by compressing to $\{i,j\}$.
		
		\textbf{Case 3.}  It remains to treat a nonsymmetric middle-nonisolated pair. Suppose
		$i\preceq k\preceq j$ with $k\ne i,j$. Choose $r,s\in\DD$ with
		$rs=\alpha$, and set
		\[
		A:=E_{ii}+E_{jj}+rE_{ik}+sE_{kj},
		\qquad
		B:=aE_{ii}+E_{kk}+bE_{jj}.
		\]
		By \eqref{eq:all-diagonal-coordinate},
		\[
		\Phi(B)=F_{ii}(a)E_{ii}+E_{kk}+F_{jj}(b)E_{jj}.
		\]
		If $k\preceq i$, then $k\preceq i\preceq j$ gives lower-endpoint
		nonisolation, already treated. If $j\preceq k$, then
		$i\preceq j\preceq k$ gives upper-endpoint nonisolation, also already treated.
		Thus we may assume that $i\preceq k$, $k\not\preceq i$, and
		$k\preceq j$, $j\not\preceq k$. Hence
		\begin{align*}
			P_{\{i,j\}}\Phi(A)P_{\{i,j\}}
			&=\Phi(E_{ii}+E_{jj})=E_{ii}+E_{jj},\\
			P_{\{i,k\}}\Phi(A)P_{\{i,k\}}
			&=\Phi(E_{ii}+rE_{ik})=E_{ii}+F_{ik}(r)U_{ik},\\
			P_{\{k,j\}}\Phi(A)P_{\{k,j\}}
			&=\Phi(E_{jj}+sE_{kj})=E_{jj}+F_{kj}(s)U_{kj},
		\end{align*}
		where the matrices $U_{pq}$ are as in \eqref{eq:Upq-definition}. By Lemma~\ref{lem:two-point-compressions}, these three compressions
		determine $\Phi(A)$. By Lemma~\ref{lem:chain-orientation}, the coordinates
		$(i,k)$, $(k,j)$ and $(i,j)$ have the same $\Phi$-orientation. In the $\Phi$-preserving case,
		\[
		\Phi(A)=E_{ii}+E_{jj}+F_{ik}(r)E_{ik}+F_{kj}(s)E_{kj}.
		\]
		Since
		\[
		ABA=aE_{ii}+bE_{jj}+arE_{ik}+sbE_{kj}+rsE_{ij},
		\]
		the compression identity gives
		\begin{align*}
			\Phi(aE_{ii}+bE_{jj}+\alpha E_{ij})
			&=\Phi(P_{\{i,j\}}(ABA)P_{\{i,j\}})\\
			&=P_{\{i,j\}}\Phi(A)\Phi(B)\Phi(A)P_{\{i,j\}}\\
			&=F_{ii}(a)E_{ii}+F_{jj}(b)E_{jj}
			+F_{ik}(r)F_{kj}(s)E_{ij}\\
			&=F_{ii}(a)E_{ii}+F_{jj}(b)E_{jj}
			+F_{ij}(\alpha)E_{ij},
		\end{align*}
		where the last equality follows from \eqref{eq:chain-coeff-preserving} and
		$rs=\alpha$. In the $\Phi$-reversing case, one easily gets
		\[
		\Phi(A)=E_{ii}+E_{jj}+F_{ik}(r)E_{ki}+F_{kj}(s)E_{jk},
		\]
		and the same computation gives
		\begin{align*}
			\Phi(aE_{ii}+bE_{jj}+\alpha E_{ij})
			&=P_{\{i,j\}}\Phi(A)\Phi(B)\Phi(A)P_{\{i,j\}}\\
			&=F_{ii}(a)E_{ii}+F_{jj}(b)E_{jj}
			+F_{kj}(s)F_{ik}(r)E_{ji}\\
			&=F_{ii}(a)E_{ii}+F_{jj}(b)E_{jj}
			+F_{ij}(\alpha)E_{ji},
		\end{align*}
		by \eqref{eq:chain-coeff-reversing} and $rs=\alpha$. This closes the proof.
	\end{proof}
	
	\subsection{Additivity}\label{sec:additivity}
	Before treating the general case, we first focus on the $M_2(\DD)$ case and resolve it completely, by carefully adapting Step 1 of the proof of \cite[Theorem~1]{Lesnjak-Sze} to the noncommutative setting. Throughout,  if $\theta:\DD\to\DD$ is a map and
	$A=[a_{ij}]\in M_n(\DD)$, we denote by $\theta[A]$ the matrix obtained from
	$A$ by applying $\theta$ entrywise, that is,
	\[
	\theta[A]:=[\theta(a_{ij})]\in M_n(\DD).
	\]
	
	\begin{remark}\label{rem:transpose-convention}
		For $n\ge1$, the ordinary transpose on $M_n(\DD)$ is anti-multiplicative if
		and only if $\DD$ is commutative. Indeed, if $(AB)^t=B^tA^t$ for all
		$A,B\in M_n(\DD)$, then applying this to $A=aE_{11}$ and $B=bE_{11}$ gives
		$ab=ba$ for all $a,b\in\DD$. Consequently, over a noncommutative division
		ring, the product-reversing standard form has to be written as a \emph{twisted transpose}. More precisely, if $\tau$ is an anti-endomorphism of $\DD$, then $A\mapsto \tau[A^t]$ is an anti-endomorphism of $M_n(\DD)$. Indeed, for $A=[a_{ij}]$ and $B=[b_{ij}]$ in $M_n(\DD)$, we have
		\[
		(\tau[(AB)^t])_{ij}
		=
		\tau\left(\sum_{k=1}^n a_{jk}b_{ki}\right)
		=
		\sum_{k=1}^n \tau(b_{ki})\tau(a_{jk})
		=
		(\tau[B^t]\tau[A^t])_{ij}.
		\]
		In particular, if $\DD$ admits no injective anti-endomorphism, the reversing case in the next proposition is void.
	\end{remark}
	
	\begin{proposition}\label{prop:m2-rigidity}
		Every injective Jordan semi-triple map $\psi:M_2(\DD)\to M_2(\DD)$ is additive. More precisely, there are $T\in \GL_2(\DD)$ and $\eps\in\{1,-1\}$ such that either
		\begin{equation}\label{eq:psi-M2-mult}
			\psi(A)=\eps\, T\sigma[A] T^{-1},
			\qquad A\in M_2(\DD),
		\end{equation}
		for a ring monomorphism $\sigma:\DD\to\DD$, or
		\begin{equation}\label{eq:psi-M2-antimult}
			\psi(A)=\eps\, T \tau[A^t] T^{-1},
			\qquad A\in M_2(\DD),
		\end{equation}
		for a ring anti-monomorphism $\tau:\DD\to\DD$.
	\end{proposition}
	
	\begin{proof}
		Following Lemma~\ref{lem:normalization} and  Proposition~\ref{prop:diagonal-straightening}, from $\psi$ we obtain the normalized injective Jordan semi-triple map $\Psi : M_2(\DD) \to M_2(\DD)$ which in particular satisfies the properties
		\[
		\Psi(0)=0,
		\qquad \Psi(I)=I,
		\qquad \Psi(E_{11})=E_{11},
		\qquad \Psi(E_{22})=E_{22}.
		\]
		By Lemma~\ref{lem:coordinate-setup}, for all $a \in \DD$ we have
		\[
		\Psi(aE_{12})=F_{12}(a)U_{12},
		\qquad
		\Psi(aE_{21})=F_{21}(a)U_{21},
		\]
		where $U_{12},U_{21}$ are as in \eqref{eq:Upq-definition}. The identity $E_{12}E_{21}E_{12}=E_{12}$ excludes $U_{12}=U_{21}$, so both coordinates are $\Psi$-preserving or both are $\Psi$-reversing. If $\alpha:=F_{12}(1)$ and $\beta:=F_{21}(1)$, the same identity gives $\alpha\beta\alpha=\alpha$, hence $\beta=\alpha^{-1}$. Further replace $\Psi$ by $D^{-1}\Psi(\cdot)D$, where $D \in \GL_2(\DD)$ is a diagonal similarity fixing $E_{11}$ and $E_{22}$ which normalizes these coefficients: use $D = \diag(1,\alpha^{-1})$ in the $\Psi$-preserving case and $D = \diag(1,\alpha)$ in the $\Psi$-reversing case. Continuing to use the notation of Lemma~\ref{lem:coordinate-setup}, write the four coordinate maps again as $F_{11},F_{22},F_{12},F_{21}$ (so that $F_{12}(1)= F_{21}(1) = 1$). 
		
		Assume first that the orientation is $\Psi$-preserving. Let $A=[a_{pq}] \in M_2(\DD)$ be arbitrary. The identities
		\begin{equation}\label{eq:the-four-sandwich-identities}
			E_{11}AE_{11}=a_{11}E_{11},
			\quad
			E_{22}AE_{22}=a_{22}E_{22},
			\quad
			E_{12}AE_{12}=a_{21}E_{12},
			\quad
			E_{21}AE_{21}=a_{12}E_{21}
		\end{equation}
		give, after applying $\Psi$ and using the known images of $E_{11},E_{22},E_{12},E_{21}$,
		\begin{equation}\label{eq:entry-formula}
			\Psi(A) = \Psi\left(\begin{bmatrix}
				a_{11} & a_{12} \\ a_{21} & a_{22}
			\end{bmatrix}\right) = \begin{bmatrix}
				F_{11}(a_{11}) & F_{21}(a_{12}) \\
				F_{12}(a_{21}) & F_{22}(a_{22}) 
			\end{bmatrix}.
		\end{equation}
		For instance, $E_{21}\Psi(A)E_{21}=\Psi(A)_{12}E_{21}$, which explains the placement of $F_{21}$ in the $(1,2)$-entry. If $A=aE_{12}$ with $a \in \DD$, the formula \eqref{eq:pure-preserving} from Lemma~\ref{lem:coordinate-setup} gives $\Psi(A)=F_{12}(a)E_{12}$, while the displayed entry formula \eqref{eq:entry-formula} gives $\Psi(A)=F_{21}(a)E_{12}$. Therefore $F_{21}(a)=
		F_{12}(a)$. Thus we may denote $f:= F_{12}=F_{21}$. Let
		\[
		V:=E_{11}+E_{12}+E_{21}+E_{22} = \begin{bmatrix}
			1 & 1 \\ 1 & 1
		\end{bmatrix}.
		\]
		Applying the entry formula \eqref{eq:entry-formula} to $A=V$ shows that $\Psi(V)=V$. The identity $V(aE_{11})V=aV$ gives
		\[
		F_{11}(a)V = VF_{11}(a)E_{11}V=\Psi(aV)  \stackrel{\eqref{eq:entry-formula}}= \begin{bmatrix}
			F_{11}(a) & f(a) \\ f(a) & F_{22}(a)
		\end{bmatrix}.
		\]
		Comparing entries gives $F_{11}(a)=F_{22}(a)=f(a)$ for all $a \in \DD$. Hence $\Psi(A)=\theta[A]$ for a single injective unital map $\theta:\DD\to\DD$.
		
		In the $\Psi$-reversing case we have $\Psi(E_{12})=E_{21}$ and $\Psi(E_{21})=E_{12}$. The same four sandwich identities \eqref{eq:the-four-sandwich-identities} give
		\[
		\Psi(A) = \Psi\left(\begin{bmatrix}
			a_{11} & a_{12} \\ a_{21} & a_{22}
		\end{bmatrix}\right) = \begin{bmatrix}
			F_{11}(a_{11}) & F_{12}(a_{21}) \\
			F_{21}(a_{12}) & F_{22}(a_{22}) 
		\end{bmatrix}.
		\]
		Comparing this entry formula with \eqref{eq:pure-reversing} again gives $F_{12}=F_{21}=:f$. Also $\Psi(V)=V$, and the identity $V(aE_{11})V=aV$ gives $F_{11}=F_{22}=f$. Hence $\Psi(A)=\theta[A^t]$ for a single injective unital map $\theta:\DD\to\DD$.
		
		It remains to show the required algebraic properties of $\theta$. For $a,b\in \DD$, let $X:=aE_{11}+bE_{12}$. Assume the $\Psi$-preserving case. Since $\Psi(I) = I$, it follows $\Psi(X^2)=\Psi(X)^2$. As $X^2=a^2E_{11}+abE_{12}$, comparison of the $E_{12}$-coefficient gives $\theta(ab)=\theta(a)\theta(b)$. Applying $\Psi$ to $V(aE_{11}+bE_{12})V=(a+b)V$ and comparing any entry gives $\theta(a+b)=\theta(a)+\theta(b)$. Thus $\theta$ is a ring monomorphism. In this case, set $\sigma:=\theta$. In the $\Psi$-reversing case the same calculation gives
		\[
		\Psi(X)=\theta(a)E_{11}+\theta(b)E_{21},
		\qquad
		\Psi(X)^2=\theta(a)^2E_{11}+\theta(b)\theta(a)E_{21},
		\]
		so $\theta(ab)=\theta(b)\theta(a)$. In the same fashion we obtain additivity. Hence $\theta$ is a ring
		anti-monomorphism. In this case, set $\tau:=\theta$.
		
		It remains to undo the reductions. Let $J:=\psi(I)$. By
		Lemma~\ref{lem:normalization}, the initial normalized map is given by $\Phi(A):=J\psi(A)$, where $J^2=I_2$, and $J$ commutes with $\Phi(M_2(\DD))$. Let $T\in\GL_2(\DD)$ be the
		product of all range similarities used after this normalization, chosen so
		that the final normalized map is
		\[
		\Psi(A):=T^{-1}\Phi(A)T=T^{-1}J\psi(A)T, \qquad A \in M_2(\DD).
		\]
		Set $\widehat J:=T^{-1}JT$. Then $\widehat J$ commutes with $\Psi(M_2(\DD))$. In either the $\Psi$-preserving
		or the $\Psi$-reversing case, the image $\Psi(M_2(\DD))$ contains the four standard
		matrix units, and hence $\widehat J=\eps I_2$ for some $\eps\in\DD$. Since
		$\widehat J^2=I_2$, we have $\eps^2=1$ and thus $\eps=\pm 1$. 
		
		In the $\Psi$-preserving case, we have $\Psi(A)=\sigma[A]$ for all
		$A\in M_2(\DD)$. Since
		\[
		\Psi(A)=T^{-1}J\psi(A)T=\eps\,T^{-1}\psi(A)T,
		\]
		we obtain
		\[
		\psi(A)=\eps\,T\sigma[A]T^{-1},
		\qquad A\in M_2(\DD).
		\]
		Thus \eqref{eq:psi-M2-mult} holds. Similarly, in the $\Psi$-reversing case, $\Psi(A)=\tau[A^t]$ for all
		$A\in M_2(\DD)$, and the same calculation gives
		\[
		\psi(A)=\eps\,T\tau[A^t]T^{-1},
		\qquad A\in M_2(\DD).
		\]
		Thus \eqref{eq:psi-M2-antimult} holds. This proves the claim.
	\end{proof}
	
	We now return to an arbitrary injective Jordan semi-triple map
	$\phi:\R_{\preceq}\to M_n(\DD)$ and to its associated normalized map $\Phi$. As before, we simply write  $\R=\R_{\preceq}$.
	
	\begin{proposition}\label{prop:suff-additivity}
		Suppose that every nonsymmetric comparable pair is nonisolated and that, whenever singleton connected components of $\Gamma_{\preceq}$ occur, $\DD\cong\FF_3$ or $\DD\cong\FF_4$. Then $\phi$ is additive.
	\end{proposition}
	
	\begin{proof}
		First consider a pair of indices $i \prec j$ on which Proposition~\ref{prop:nonisolated-two-point} applies. For $a,b\in \DD$, let
		\[
		X:=aE_{ii}+bE_{jj}+E_{ij},
		\]
		so that
		\[
		\Phi(X) = F_{ii}(a)E_{ii} + F_{jj}(b)E_{jj} + F_{ij}(1)U_{ij}.
		\]
		We have
		\[
		\Phi(X^2) = \Phi(a^2E_{ii} + b^2E_{jj} + (a+b)E_{ij}) = F_{ii}(a^2)E_{ii} + F_{jj}(b^2)E_{jj} + F_{ij}(a+b)U_{ij}.
		\]
		Since $\Phi(I)=I$, we have $\Phi(X^2)=\Phi(X)^2$. If $(i,j)$ is $\Phi$-preserving, we have
		\[
		\Phi(X^2) = \Phi(X)^2 = F_{ii}(a)^2E_{ii} + F_{jj}(b)^2E_{jj} + (F_{ii}(a)F_{ij}(1)+F_{ij}(1)F_{jj}(b))E_{ij}.
		\]
		Comparing the $E_{ij}$-coefficient gives
		\[
		F_{ij}(a+b)=F_{ii}(a)F_{ij}(1)+F_{ij}(1)F_{jj}(b)
		\stackrel{\eqref{eq:sandwich-preserving}}=F_{ij}(a)+F_{ij}(b).
		\]
		Therefore, $F_{ij}$ is additive. Since $F_{ij}(a)=F_{ii}(a)F_{ij}(1)=F_{ij}(1)F_{jj}(a)$ and $F_{ij}(1)\ne0$, cancellation gives additivity of $F_{ii}$ and $F_{jj}$. In the $\Phi$-reversing case the same calculation gives
		\[
		F_{ij}(a+b)=F_{ij}(1)F_{ii}(a)+F_{jj}(b)F_{ij}(1)
		\stackrel{\eqref{eq:sandwich-reversing}}=F_{ij}(a)+F_{ij}(b),
		\]
		and cancellation again gives additivity of $F_{ii}$ and $F_{jj}$. Consequently, for every coordinate $(i,j)$ on which Proposition~\ref{prop:nonisolated-two-point} applies, the maps $F_{ii},F_{ij}$ and $F_{jj}$ are additive.
		
		We now prove additivity of $\Phi$ on all of $\R$. Let $i\ne j$ and set
		$P:=P_{\{i,j\}}$. Given $X,Y\in\R$, Lemma~\ref{lem:two-point-compressions}
		shows that it is enough to prove
		\[
		P\Phi(X+Y)P=P(\Phi(X)+\Phi(Y))P.
		\]
		By the compression identity,
		\[
		P\Phi(X+Y)P=\Phi(PXP+PYP),
		\qquad
		P(\Phi(X)+\Phi(Y))P=\Phi(PXP)+\Phi(PYP).
		\]
		Thus it remains to prove additivity on the two-point corner $P\R P$.
		
		If exactly one of $i\preceq j$ and $j\preceq i$ holds, then, after
		interchanging $i$ and $j$ if necessary, we may assume that
		$i\preceq j$ and $j\not\preceq i$. The pair is nonisolated by the assumption, so
		Proposition~\ref{prop:nonisolated-two-point} applies. Writing
		\[
		PXP=x_{ii}E_{ii}+x_{jj}E_{jj}+x_{ij}E_{ij},
		\qquad
		PYP=y_{ii}E_{ii}+y_{jj}E_{jj}+y_{ij}E_{ij},
		\]
		we obtain
		\begin{align*}
			\Phi(PXP+PYP)
			&=
			F_{ii}(x_{ii}+y_{ii})E_{ii}
			+F_{jj}(x_{jj}+y_{jj})E_{jj}
			+F_{ij}(x_{ij}+y_{ij})U_{ij}  \\
			&=
			\Phi(PXP)+\Phi(PYP),
		\end{align*}
		from Proposition~\ref{prop:nonisolated-two-point} and the additivity of
		$F_{ii},F_{jj},F_{ij}$.
		
		On the other hand, if $i\asymp j$, then $P\R P\cong M_2(\DD)$. By the compression identity, $\Phi(P\R P)\subseteq PM_n(\DD)P$, and the
		restriction
		\[
		P\R P\to PM_n(\DD)P,\qquad Z\mapsto\Phi(Z),
		\]
		is an injective Jordan semi-triple map. Proposition~\ref{prop:m2-rigidity}
		gives additivity on this corner.
		
		It remains to consider the case where $i$ and $j$ are incomparable. Then $P\R P=\DD E_{ii}\oplus\DD E_{jj}.$ We prove that $F_{ii}$ and $F_{jj}$ are additive. If $\{i\}$ is a singleton connected component of $\Gamma_{\preceq}$, then $F_{ii}$ is
		an injective scalar Jordan semi-triple map. Indeed, for $a,b \in \DD$, using \eqref{eq:diag-coordinate} we obtain
		\[
		F_{ii}(aba)E_{ii}
		=
		\Phi(abaE_{ii})
		=
		\Phi(aE_{ii})\Phi(bE_{ii})\Phi(aE_{ii})
		=
		F_{ii}(a)F_{ii}(b)F_{ii}(a)E_{ii}.
		\]
		Hence $F_{ii}$ is additive by
		Proposition~\ref{prop:scalar-classification} and the assumption on $\DD$. Otherwise,
		$i$ is comparable with some $k\ne i$. If $i\asymp k$, then
		$P_{\{i,k\}}\R P_{\{i,k\}}\cong M_2(\DD)$, and
		Proposition~\ref{prop:m2-rigidity} gives the additivity of $F_{ii}$. If the
		comparability between $i$ and $k$ is nonsymmetric, then the corresponding
		nonsymmetric pair is nonisolated by the assumption, and the preceding local
		additivity argument gives the additivity of $F_{ii}$. The same proof gives
		the additivity of $F_{jj}$.
		
		Now write
		\[
		PXP=x_{ii}E_{ii}+x_{jj}E_{jj},
		\qquad
		PYP=y_{ii}E_{ii}+y_{jj}E_{jj}.
		\]
		Using \eqref{eq:two-diagonal-coordinate} and the additivity of
		$F_{ii},F_{jj}$ gives
		\[
		\Phi(PXP+PYP)= F_{ii}(x_{ii}+y_{ii})E_{ii} + F_{jj}(x_{jj}+y_{jj})E_{jj}=\Phi(PXP)+\Phi(PYP).
		\]
		Thus $\Phi$ is additive. Lemma~\ref{lem:normalization} implies that $\phi$ is additive as well.
	\end{proof}
	
	\begin{proof}[Proof of Theorem~\ref{thm:main}]
		Assume first that $(\mathsf P)$ and $(\mathsf S)$ hold. Then
		Proposition~\ref{prop:suff-additivity} implies that every injective Jordan
		semi-triple map $\phi:\R \to M_n(\DD)$ is additive.
		
		If $(\mathsf S)$ fails,
		then $\Gamma_{\preceq}$ has a singleton connected component, say $\{i\}$, and
		$\DD\not\cong\FF_3,\FF_4$. Since $|Z(\DD)|>2$, also
		$\DD\not\cong\FF_2$. Hence, by Proposition~\ref{prop:scalar-classification},
		there exists a nonadditive injective scalar Jordan semi-triple map
		$f:\DD\to\DD$. By Lemma~\ref{lem:central-components}, $\R$ decomposes as the direct sum of
		two-sided ideals
		\[
		\R
		=
		\DD E_{ii}
		\oplus
		\left(P_{[n]\setminus\{i\}}\R P_{[n]\setminus\{i\}}\right).
		\]
		Define a map $\Theta:\R\to\R$ by letting it act as $f$ on $\DD E_{ii}$ and as the
		identity on $P_{[n]\setminus\{i\}}\R P_{[n]\setminus\{i\}}$. Since the two
		summands annihilate each other, $\Theta$ is a nonadditive injective Jordan
		semi-triple map on $\R$. 
		
		On the other hand, if $(\mathsf P)$ fails, then there exists an isolated nonsymmetric comparable pair. Proposition~\ref{prop:bad-interval} gives a nonadditive
		injective Jordan semi-triple self-map on $\R$. Therefore, the global automatic additivity assumption implies both $(\mathsf P)$ and $(\mathsf S)$.
	\end{proof}
	
	\section{The role of the centre-size hypothesis}\label{sec:centre-size-hypothesis}
	
	This section clarifies the role of the centre-size hypothesis in
	Theorem~\ref{thm:main}. The sufficiency direction needs only the weaker
	assumption $|\DD|>2$: if $(\mathsf P)$ and $(\mathsf S)$ hold, then every
	injective Jordan semi-triple map $\R_{\preceq}\to M_n(\DD)$ is additive. The
	stronger assumption $|Z(\DD)|>2$ is needed in the necessity direction,
	specifically in the proof that automatic additivity forces $(\mathsf P)$:
	Proposition~\ref{prop:bad-interval} deforms an isolated nonsymmetric
	comparable pair using a central scalar different from $0$ and $1$.
	
	The examples below show that neither restriction can be removed. First, over
	$\FF_2$, the conditions $(\mathsf P)$ and $(\mathsf S)$ need not imply
	automatic additivity. Second, for division rings with $|\DD|>2$ but
	$|Z(\DD)|=2$, automatic additivity may hold even when $(\mathsf P)$ fails.
	
	\begin{example}\label{ex:tn-f2-final}
		Consider the upper-triangular ring $T_n(\FF_2)$, where $n \ge 2$. Define a map
		\[
		\Psi: T_n(\FF_2)\to T_n(\FF_2),
		\qquad
		\Psi(X):=X+x_{11}(1+x_{nn})E_{1n}.
		\]
		The map $\Psi$ is injective. Indeed, the diagonal entries are fixed, and once
		$x_{11}$ and $x_{nn}$ are known, the original $(1,n)$-entry is recovered from
		the $(1,n)$-entry of $\Psi(X)$. It is not additive, because
		\[
		\Psi(E_{11})=E_{11}+E_{1n},\qquad
		\Psi(E_{nn})=E_{nn},\qquad
		\Psi(E_{11}+E_{nn})=E_{11}+E_{nn}.
		\]
		We verify the Jordan semi-triple identity. Let $X=[x_{ij}],Y=[y_{ij}]\in T_n(\FF_2)$. Then
		\[
		\Psi(X)=X+x_{11}(1+x_{nn})E_{1n},
		\qquad
		\Psi(Y)=Y+y_{11}(1+y_{nn})E_{1n}.
		\]
		Since $E_{1n}AE_{1n}=0$ for every $A\in T_n(\FF_2)$, all terms containing two copies
		of $E_{1n}$ vanish. Moreover,
		\[
		E_{1n}YX=y_{nn}x_{nn}E_{1n},
		\qquad
		XYE_{1n}=x_{11}y_{11}E_{1n},
		\qquad
		XE_{1n}X=x_{11}x_{nn}E_{1n}.
		\]
		Hence
		\begin{align*}
			\Psi(X)\Psi(Y)\Psi(X)
			&=
			XYX+ \big(x_{11}(1+x_{nn})y_{nn}x_{nn}
			+x_{11}y_{11}x_{11}(1+x_{nn}) \\
			& \hspace{1.5cm} +x_{11}y_{11}(1+y_{nn})x_{nn}
			\big)E_{1n}  \\
			&=
			XYX+x_{11}y_{11}(1+x_{nn}y_{nn})E_{1n},
		\end{align*}
		where we use $r^2=r$ for $r\in\FF_2$. The first and last diagonal entries of
		$XYX$ are $x_{11}y_{11}$ and $x_{nn}y_{nn}$, respectively. Therefore
		\[
		\Psi(XYX)
		=
		XYX+x_{11}y_{11}(1+x_{nn}y_{nn})E_{1n}
		=
		\Psi(X)\Psi(Y)\Psi(X).
		\]
		Thus $\Psi$ is a nonadditive injective Jordan semi-triple map on $T_n(\FF_2)$. 
		
		Since the comparability graph of $T_n(\FF_2)$ is connected and nonsingleton,
		condition $(\mathsf S)$ is vacuous for $T_n(\FF_2)$. On the other hand, by
		Example~\ref{ex:T_n-SMR}, condition $(\mathsf P)$ holds for $T_n(\FF_2)$
		whenever $n\ge3$. Thus, for $n\ge3$, both $(\mathsf P)$ and $(\mathsf S)$
		hold for $T_n(\FF_2)$, while automatic additivity nevertheless fails. This
		shows that the centre-size hypothesis in Theorem~\ref{thm:main} cannot be
		omitted.
	\end{example}
	
	\begin{example}\label{ex:centre-F2-isolated-additive}
		Let $\DD$ be the universal field of fractions of the free associative
		$\FF_2$-algebra $\FF_2\langle x,y\rangle$. Thus $\DD$ is the free division
		ring over $\FF_2$ generated by two noncommuting elements $x$ and $y$. 
		
		For $S\subseteq\DD$, let
		\[
		C_\DD(S):=\{d\in\DD:ds=sd\text{ for all }s\in S\}
		\]
		be its centralizer in $\DD$. We use Cohn's centralizer theorem for free division rings \cite[Corollaire~2]{Cohn-free-centralizers}: the centralizer in $\DD$ of
		every element of $\DD\setminus\FF_2$ is commutative. In particular,
		$Z(\DD)=\FF_2$, since otherwise a non-scalar central element would have
		centralizer equal to the noncommutative division ring $\DD$. Consequently, if $S\subseteq\DD$ contains at least two noncommuting elements, then
		\begin{equation}\label{eq:centre-F2-subset-centralizer}
			C_\DD(S)=\FF_2 .
		\end{equation}
		Indeed, since $Z(\DD)=\FF_2$, we have $\FF_2\subseteq C_\DD(S)$. Conversely, if $\lambda\in C_\DD(S)\setminus\FF_2$, then $C_\DD(\{\lambda\})$ is commutative by Cohn's theorem. Since $\lambda$ commutes with every element of
		$S$, we have $S\subseteq C_\DD(\{\lambda\})$, a contradiction.
		
		We consider the $2 \times 2$ upper-triangular ring $T_2(\DD)$. As already noted (Example~\ref{ex:T_n-SMR}), for $T_2(\DD)$ the property $(\mathsf P)$ fails and $(\mathsf S)$ is vacuous. We claim
		that, nevertheless, every injective Jordan semi-triple map
		$T_2(\DD)\to M_2(\DD)$ is additive. By Lemma~\ref{lem:normalization} and
		Proposition~\ref{prop:diagonal-straightening}, it is enough to establish additivity for a 
		normalized map $\Phi :T_2(\DD)\to M_2(\DD)$ satisfying
		\[
		\Phi(0)=0,\qquad
		\Phi(I)=I,\qquad
		\Phi(E_{11})=E_{11},\qquad
		\Phi(E_{22})=E_{22}.
		\]
		For $X=aE_{11}+bE_{12}+cE_{22}$ with $a,b,c\in \DD$, the identities
		$E_{11}XE_{11}=aE_{11}$ and $E_{22}XE_{22}=cE_{22}$, together with the
		Jordan semi-triple identity, imply
		\begin{equation}\label{eq:diagonal-corners-D}
			E_{11}\Phi(X)E_{11}=\Phi(aE_{11}),\qquad
			E_{22}\Phi(X)E_{22}=\Phi(cE_{22}).
		\end{equation}
		Moreover,
		\[
		\Phi(aE_{11})=E_{11}\Phi(aE_{11})E_{11},\qquad
		\Phi(cE_{22})=E_{22}\Phi(cE_{22})E_{22}.
		\]
		Hence there are injective maps $\alpha,\delta:\DD\to\DD$, with
		$\alpha(0)=\delta(0)=0$ and $\alpha(1)=\delta(1)=1$, and maps $\beta,\gamma:\DD^3\to\DD$ such that
		\begin{equation}\label{eq:centre-F2-coordinate-form}
			\Phi(aE_{11}+bE_{12}+cE_{22})
			=
			\alpha(a)E_{11}+\beta(a,b,c)E_{12}+\gamma(a,b,c)E_{21}+\delta(c)E_{22},
		\end{equation}
		for all $a,b,c \in \DD$.
		
		We first show that globally $\beta \equiv 0$ or $\gamma \equiv 0$. Since $(bE_{12})I(bE_{12})=0$, while $\Phi(0)=0$ and $\Phi(I)=I$, we have $\Phi(bE_{12})^2=0$. By \eqref{eq:diagonal-corners-D}, the diagonal corners of $\Phi(bE_{12})$ vanish, so by
		\eqref{eq:centre-F2-coordinate-form},
		\[
		\Phi(bE_{12})=\beta(0,b,0)E_{12}+\gamma(0,b,0)E_{21}.
		\]
		If $b\ne0$, then $\Phi(bE_{12})\ne0$, and
		\[
		\Phi(bE_{12})^2
		=
		\beta(0,b,0)\gamma(0,b,0)E_{11}
		+
		\gamma(0,b,0)\beta(0,b,0)E_{22}.
		\]
		Therefore, exactly one of $\beta(0,b,0)$ and
		$\gamma(0,b,0)$ is nonzero. Choose $b_0\in\DD^\times$. If
		$\Phi(b_0 E_{12})=uE_{12}$, $u\ne0$, then for every
		$X=aE_{11}+bE_{12}+cE_{22}$,
		\[
		0=\Phi((b_0 E_{12})X(b_0 E_{12}))=(uE_{12})\Phi(X)(uE_{12})
		\stackrel{\eqref{eq:centre-F2-coordinate-form}}= (u\gamma(a,b,c)u)E_{12},
		\]
		so $\gamma(a,b,c)=0$ for all $a,b,c \in \DD$. Similarly, if
		$\Phi(b_0 E_{12})=uE_{21}$, then $\beta(a,b,c)=0$ for all $a,b,c \in \DD$.
		
		\textbf{Case 1.} Assume first that $\gamma \equiv 0$. Then
		\[
		\Phi(aE_{11}+bE_{12}+cE_{22})
		=
		\alpha(a)E_{11}+\beta(a,b,c)E_{12}+\delta(c)E_{22}, \qquad a,b,c \in \DD.
		\]
		Since
		\[
		XE_{11}X=a^2E_{11}+abE_{12},\qquad
		XE_{22}X=bcE_{12}+c^2E_{22},
		\]
		comparison of the $(1,2)$-coefficients in
		$\Phi(XE_{11}X)=\Phi(X)E_{11}\Phi(X)$ and
		$\Phi(XE_{22}X)=\Phi(X)E_{22}\Phi(X)$ gives
		\begin{equation}\label{eq:centre-F2-upper-relations}
			\beta(a^2,ab,0)=\alpha(a)\beta(a,b,c),\qquad
			\beta(0,bc,c^2)=\beta(a,b,c)\delta(c).
		\end{equation}
		If $a\ne0$, then $\alpha(a) \ne 0$ and then the first relation shows that $\beta(a,b,c) = \alpha(a)^{-1}\beta(a^2,ab,0)$, which is independent of $c$. In particular, $\beta(a,b,c)=\beta(a,b,1)$. Similarly, if $c\ne0$, the second relation gives $\beta(a,b,c)=\beta(0,bc,c^2)\delta(c)^{-1}$, which is independent of $a$.
		For fixed $b$, the values $\beta(a,b,1)$ with $a\ne0$ and the values
		$\beta(1,b,c)$ with $c\ne0$ are all equal, because they both agree with
		$\beta(1,b,1)$. Thus there is a map $t:\DD\to\DD$ such that
		\[
		\beta(a,b,c)=t(b)\qquad\text{whenever }(a,c)\ne(0,0).
		\]
		Define $s:\DD\to\DD$ by $s(b):=\beta(0,b,0)$. Then
		\[
		\beta(a,b,c)=
		\begin{cases}
			t(b),& \text{ if } (a,c)\ne(0,0),\\
			s(b),& \text{ if } a=c=0.
		\end{cases}
		\]
		Moreover, $\Phi(0)=0$ and $\Phi(I)=I$ give $s(0)=t(0)=0$, and injectivity
		gives $s(1),t(1)\ne0$. Returning to \eqref{eq:centre-F2-upper-relations} gives
		\begin{equation}\label{eq:centre-F2-upper-t-relations}
			t(ab)=\alpha(a)t(b),\qquad t(bc)=t(b)\delta(c),\qquad a,b,c\in\DD.
		\end{equation}
		After replacing $\Phi$ by $X\mapsto D_0\Phi(X)D_0^{-1}$, where $D_0:=\operatorname{diag}(t(1)^{-1},1)$, and keeping the same notation for the resulting map and its coordinate functions,
		we may assume $t(1)=1$. Then \eqref{eq:centre-F2-upper-t-relations} gives $\alpha=t=\delta$. Denote this
		common map by $\theta$. Hence $\theta(ab)=\theta(a)\theta(b)$ for all $a,b\in\DD$. At this stage,
		\[
		\Phi(aE_{11}+bE_{12}+cE_{22}) = \begin{cases}
			\theta(a)E_{11} + \theta(b)E_{12} + \theta(c)E_{22}, &\text{ if } (a,c) \ne (0,0),\\
			s(b)E_{12}, &\text{ if } a=c=0.
		\end{cases}
		\]
		To prove additivity of $\theta$, take
		$X=E_{11}+u E_{12}+E_{22}$ and $Y=E_{11}+v E_{12}$. Then
		$XYX=E_{11}+(u+v)E_{12}$, and comparison of the $(1,2)$-coefficients in
		$\Phi(XYX)=\Phi(X)\Phi(Y)\Phi(X)$ gives $\theta(u+v)=\theta(u)+\theta(v)$  for all $u,v\in\DD$. Thus $\theta$ is a ring monomorphism. Finally, applying $\Phi$ to
		\[
		(aE_{11}+cE_{22})(bE_{12})(aE_{11}+cE_{22})
		=
		abc\,E_{12}, \qquad a,b,c\in\DD,
		\]
		gives the identity
		\[
		s(abc)=\theta(a)s(b)\theta(c), \qquad a,b,c\in\DD.
		\]
		For $\lambda:=s(1)$, we have
		$s(a)=\theta(a)\lambda=\lambda\theta(a)$ for all $a\in\DD$, so
		$\lambda\in C_\DD(\theta(\DD))$. Since $\theta$ is injective and $\DD$ is
		noncommutative, $\theta(\DD)$ is noncommutative. Hence
		\eqref{eq:centre-F2-subset-centralizer}, applied to $S:=\theta(\DD)$, gives
		$\lambda\in\FF_2$. As $\lambda\ne0$, we get $\lambda=1$. Thus
		$s=t=\theta$, which shows that $\Phi$ is additive in the case $\gamma\equiv0$.
		
		\textbf{Case 2.} Assume now that $\beta\equiv0$. This case can be handled by
		arguments similar to those in Case~1. Alternatively, we reduce it to
		Case~1.  We first recall that $\DD$ admits an involution.  Set $\mathcal{A}:=\FF_2\langle x,y\rangle$, and regard $\DD$ as the
		universal field of fractions of $\mathcal{A}$. The word-reversal map is an involutive anti-automorphism $\rho$ of $\mathcal A$: on monomials it reverses the order of the letters, and it is
		extended $\FF_2$-linearly. By \cite[Corollary~2.5.2]{Cohn-free-ideal-rings}, $\mathcal A$ is a free ideal ring. Hence, by \cite[Theorem]{Klein-involutorial}, $\rho$ extends uniquely to
		an involutive anti-automorphism $\tau$ of the universal field of fractions $\DD$.
		
		By Remark~\ref{rem:transpose-convention}, the associated twisted transpose
		$\kappa:M_2(\DD)\to M_2(\DD)$, $\kappa(Z):=\tau[Z^t]$, is a ring anti-automorphism of $M_2(\DD)$ and fixes $I,E_{11},E_{22}$.  Define a map
		\[
		\Psi:T_2(\DD)\to M_2(\DD),
		\qquad
		\Psi(X):=\kappa(\Phi(X)).
		\]
		Since $\kappa$ is bijective and fixes $I,E_{11},E_{22}$, the map $\Psi$ is
		normalized and injective. Moreover, $\Psi$ is a Jordan semi-triple map,
		because $\kappa$ reverses products and the two outer factors in the Jordan
		semi-triple identity coincide. As $\beta\equiv0$, we have
		\[
		\Phi(aE_{11}+bE_{12}+cE_{22})
		=
		\alpha(a)E_{11}+\gamma(a,b,c)E_{21}+\delta(c)E_{22},
		\qquad a,b,c\in\DD.
		\]
		Therefore
		\[
		\Psi(aE_{11}+bE_{12}+cE_{22})
		=
		\tau(\alpha(a))E_{11}
		+
		\tau(\gamma(a,b,c))E_{12}
		+
		\tau(\delta(c))E_{22}.
		\]
		Thus $\Psi$ is of the form treated in Case~1, and hence $\Psi$ is additive.
		Since $\kappa$ is an additive bijection, it follows that
		$\Phi=\kappa^{-1}\circ\Psi$ is additive as well.
		
		It follows that the normalized map $\Phi$ is additive. Thus $T_2(\DD)$ has an
		isolated nonsymmetric comparable pair and no singleton connected component,
		but every injective Jordan semi-triple map $T_2(\DD)\to M_2(\DD)$ is
		additive. This shows that the converse direction of Theorem~\ref{thm:main}
		genuinely needs the centre-size assumption $|Z(\DD)|>2$, not merely
		$|\DD|>2$.
	\end{example}
	
	\section{Standard form of additive injective Jordan semi-triple maps}\label{sec:standard-form}
	
	In this final section we describe the standard form of additive
	injective Jordan semi-triple maps on general SMRs. This complements
	Theorem~\ref{thm:main}: the theorem gives a criterion for automatic
	additivity, whereas here additivity is assumed and the resulting
	structure is described along the connected components of the
	comparability graph.
	
	Throughout this section, let $\R=\R_{\preceq}\subseteq M_n(\DD)$ be an SMR,
	with $|\DD|>2$, and let $\phi:\R\to M_n(\DD)$ be an additive injective Jordan semi-triple map. Let $\Phi : \R \to M_n(\DD)$ be the corresponding normalized map from Lemma~\ref{lem:normalization}. After conjugating
	the range as in Proposition~\ref{prop:diagonal-straightening}, we may assume
	\[
	\Phi(I)=I,
	\qquad
	\Phi(P_S)=P_S\quad(S\subseteq[n]).
	\]
	Since $\Phi$ is additive and satisfies the Jordan semi-triple identity, it is
	a unital injective Jordan homomorphism in the square-preserving sense:
	\begin{equation}\label{eq:square-after-add}
		\Phi(X^2)=\Phi(X)^2,\qquad \text{for all }X\in\R .
	\end{equation}
	For an injective endomorphism or anti-endomorphism $\theta:\DD\to\DD$, set
	\[
	\Cent(\theta):=\{c\in\DD:c\theta(a)=\theta(a)c
	\text{ for all }a\in\DD\},
	\qquad
	\Cent(\theta)^\times:=\Cent(\theta)\setminus\{0\}.
	\]
	If $C$ is a connected component of $\Gamma_{\preceq}$, we write $\preceq_C$
	for the restriction of $\preceq$ to $C$, and $\preceq_C^t$ for the reverse
	preorder. Following Coelho~\cite{Coelho-LAA}, if $\rho$ is a preorder on a set $S$, a
	map $g:\rho\to\DD^\times$ is called \emph{transitive} if
	\[
	g(x,z)=g(x,y)g(y,z),
	\qquad \text{for all }x,y,z \in S \text{ with } (x,y), (y,z) \in \rho.
	\]
	In particular, $g(x,x)=1$ for all $x \in S$. In the SMR setting, transitive maps enter as coefficient multipliers, and the defining relation is exactly the compatibility condition forced by the
	matrix-unit products $E_{ij}E_{jk}=E_{ik}$.
	
	If $\sigma:\DD\to\DD$ is an injective endomorphism and
	$g:\mathop{\preceq_C} \to\Cent(\sigma)^\times$ is transitive, define
	\begin{equation}\label{eq:Theta-sigma}
		\Theta^M_{\sigma,g} : \R_{\preceq_C} \to \R_{\preceq_C}, \qquad \Theta^M_{\sigma,g}
		\left(\sum_{\substack{i\preceq j\\ i,j\in C}}x_{ij}E_{ij}\right)
		:=
		\sum_{\substack{i\preceq j\\ i,j\in C}}
		g(i,j)\sigma(x_{ij})E_{ij}.
	\end{equation}
	If $\tau:\DD\to\DD$ is an injective anti-endomorphism and
	$h:\mathop{\preceq_C}\to\Cent(\tau)^\times$ is transitive, define
	\begin{equation}\label{eq:Theta-tau}
		\Theta^A_{\tau,h} : \R_{\preceq_C} \to \R_{\preceq_C^t}, \qquad \Theta^A_{\tau,h}
		\left(\sum_{\substack{i\preceq j\\ i,j\in C}}x_{ij}E_{ij}\right)
		:=
		\sum_{\substack{i\preceq j\\ i,j\in C}}
		h(i,j)^{-1}\tau(x_{ij})E_{ji}.
	\end{equation}
	Since the values of $g$ and $h$ belong to $\Cent(\sigma)$ and
	$\Cent(\tau)$, respectively, transitivity gives, for all
	$X,Y\in\R_{\preceq_C}$,
	\[
	\Theta^M_{\sigma,g}(XY)
	=
	\Theta^M_{\sigma,g}(X)\Theta^M_{\sigma,g}(Y),
	\qquad
	\Theta^A_{\tau,h}(XY)
	=
	\Theta^A_{\tau,h}(Y)\Theta^A_{\tau,h}(X).
	\]
	Thus $\Theta^M_{\sigma,g}$ is an injective homomorphism
	$\R_{\preceq_C}\to\R_{\preceq_C}$, while $\Theta^A_{\tau,h}$ is an
	injective anti-homomorphism $\R_{\preceq_C}\to\R_{\preceq_C^t}$. On a
	singleton component, both constructions reduce to the scalar form
	$aE_{ii}\mapsto\theta(a)E_{ii}$, where $\theta:\DD\to\DD$ is an injective
	endomorphism or anti-endomorphism.
	
	\begin{lemma}\label{lem:additive-orientation}
		On every nonsingleton connected component $C$ of $\Gamma_{\preceq}$, all
		non-diagonal pairs $i \preceq_C j$ have the same $\Phi$-orientation, $\Phi$-preserving or
		$\Phi$-reversing.
	\end{lemma}
	
	\begin{proof}
		Fix a nonsingleton connected component $C$. Let $\mathcal G_C$ be the graph
		whose vertices are the non-diagonal pairs $(i,j)$ with $i,j\in C$ and
		$i\preceq j$. Two vertices of $\mathcal G_C$ are joined by an edge if the
		corresponding ordered pairs have at least one index in common.
		
		We first show that $\mathcal G_C$ is connected. Let $(p,q)$ and $(r,s)$ be
		two vertices of $\mathcal G_C$. Since $C$ is connected in
		$\Gamma_{\preceq}$, there is a path
		\[
		q=i_0,i_1,\ldots,i_m=r
		\]
		in $C$ such that $i_{\ell-1},i_\ell$ are distinct and comparable for all $1 \le \ell \le m$.
		For each such $\ell$, choose one of the ordered pairs
		$(i_{\ell-1},i_\ell)$ or $(i_\ell,i_{\ell-1})$ which belongs to $\preceq$.
		Adding $(p,q)$ at the beginning and $(r,s)$ at the end gives a path in
		$\mathcal G_C$ from $(p,q)$ to $(r,s)$. Thus $\mathcal G_C$ is connected. It remains to show that $\mathcal G_C$-adjacent coordinates have the same orientation. We consider cases:
		\begin{itemize}
			\item If the pairs are of the form $(i,j)$ and $(j,k)$ with $i,j,k$ distinct, then
			Lemma~\ref{lem:chain-orientation} gives the assertion.
			\item If the pairs are of the form $(i,j)$ and $(j,i)$ with $i,j$ distinct, then
			$P_{\{i,j\}}\R P_{\{i,j\}}\cong M_2(\DD)$, and
			Proposition~\ref{prop:m2-rigidity} applies.
			\item Suppose that the pairs are of the form $(i,j)$ and $(i,k)$, where $i,j,k$ are distinct. Then $(E_{ij}+E_{ik})^2=0$. Write $\Phi(E_{pq})=\lambda_{pq}U_{pq}$, where
			$\lambda_{pq}:=F_{pq}(1)\ne0$ and, as in \eqref{eq:Upq-definition}, $U_{pq}$ is either $E_{pq}$ or $E_{qp}$,
			according to the $\Phi$-orientation. If $(i,j)$ and $(i,k)$ had opposite
			$\Phi$-orientations, then, by additivity,
			$(\Phi(E_{ij}+E_{ik}))^2$ would be one of $\lambda_{ik}\lambda_{ij}E_{kj}$ and $\lambda_{ij}\lambda_{ik}E_{jk}$, both of which are nonzero. This contradicts \eqref{eq:square-after-add}.
			\item If the pairs are of the form $(j,i)$ and $(k,i)$, where $i,j,k$ are distinct, then the proof is as in the previous case.
		\end{itemize}
		Hence adjacent vertices of $\mathcal G_C$ have the same orientation. As $\mathcal G_C$ is connected, this proves the lemma.
	\end{proof}
	
	\begin{proposition}\label{prop:normalized-standard}
		Let $C$ be a nonsingleton connected component of $\Gamma_{\preceq}$. Then
		there is a diagonal matrix $D_C\in P_CM_n(\DD)P_C$, invertible in the corner
		$P_CM_n(\DD)P_C$, such that the
		map
		\[
		P_C\R P_C \to P_CM_n(\DD)P_C, \qquad X\mapsto D_C\Phi(X)D_C^{-1},
		\]
		is either $\Theta^M_{\sigma_C,g_C}$ as in \eqref{eq:Theta-sigma}, with
		$\sigma_C:\DD\to\DD$ an injective endomorphism and
		$g_C:\mathop{\preceq_C}\to\Cent(\sigma_C)^\times$ transitive, or
		$\Theta^A_{\tau_C,h_C}$ as in \eqref{eq:Theta-tau}, with
		$\tau_C:\DD\to\DD$ an injective anti-endomorphism and
		$h_C:\mathop{\preceq_C}\to\Cent(\tau_C)^\times$ transitive.
	\end{proposition}
	
	\begin{proof}
		By additivity, the compression identity \eqref{eq:compression}, and
		Lemma~\ref{lem:coordinate-setup}, the restriction of $\Phi$ to
		$P_C\R P_C$ is determined entrywise by the coordinate maps:
		\[
		\Phi\left(\sum_{\substack{i\preceq j\\ i,j\in C}}
		x_{ij}E_{ij}\right)
		=
		\sum_{\substack{i\preceq j\\ i,j\in C}}
		F_{ij}(x_{ij})U_{ij}.
		\]
		Here $U_{ii}:=E_{ii}$, and for $i\ne j$ the matrix unit $U_{ij}$ is defined
		by \eqref{eq:Upq-definition}. Put
		\[
		\lambda_{ij}:=F_{ij}(1),\qquad i,j\in C,\ i\preceq j.
		\]
		Then $\lambda_{ij}\in\DD^\times$ and $\lambda_{ii}=1$. By
		Lemma~\ref{lem:additive-orientation}, all non-diagonal coordinates in $C$
		have the same $\Phi$-orientation.
		
		Assume first that this orientation is preserving. We show that each
		$F_{ii}$, $i\in C$, is an injective endomorphism of $\DD$. Additivity,
		injectivity and unitality are already known. Since $C$ is nonsingleton, choose
		$j\in C$, $j\ne i$, comparable with $i$. If $i\preceq j$, then
		\[
		F_{ii}(ab)\lambda_{ij}
		\stackrel{\eqref{eq:sandwich-preserving}}=
		F_{ij}(ab)
		\stackrel{\eqref{eq:sandwich-preserving}}=
		F_{ii}(a)F_{ij}(b)
		\stackrel{\eqref{eq:sandwich-preserving}}=
		F_{ii}(a)F_{ii}(b)\lambda_{ij},
		\]
		and right cancellation gives $F_{ii}(ab)=F_{ii}(a)F_{ii}(b)$. If
		$j\preceq i$, the same argument applied to the coordinate $(j,i)$ gives
		\[
		\lambda_{ji}F_{ii}(ab)
		\stackrel{\eqref{eq:sandwich-preserving}}=
		F_{ji}(ab)
		\stackrel{\eqref{eq:sandwich-preserving}}=
		F_{ji}(a)F_{ii}(b)
		\stackrel{\eqref{eq:sandwich-preserving}}=
		\lambda_{ji}F_{ii}(a)F_{ii}(b),
		\]
		and left cancellation gives the same conclusion.
		
		Fix $r\in C$ and put $\sigma_C:=F_{rr}$. For each $i\in C$, choose any path $r=i_0,i_1,\ldots,i_m=i$ in $\Gamma_{\preceq}$, and choose
		one comparable direction along each edge. Define $\delta_0:=1$ and
		\[
		\delta_{\ell+1}
		:=
		\begin{cases}
			\delta_\ell\lambda_{i_\ell i_{\ell+1}},
			&\text{if } i_\ell\preceq i_{\ell+1},\\[1mm]
			\delta_\ell\lambda_{i_{\ell+1}i_\ell}^{-1},
			&\text{if } i_{\ell+1}\preceq i_\ell .
		\end{cases}
		\]
		Set $d_i:=\delta_m$. These choices are used only to construct one diagonal
		similarity; no path-independence is asserted or needed.
		
		For $p,q\in C$ with $p\preceq q$, \eqref{eq:sandwich-preserving} gives
		\begin{equation}\label{eq:diag-intertwining-preserving}
			F_{pp}(a)\lambda_{pq}=\lambda_{pq}F_{qq}(a),
			\qquad a\in\DD .
		\end{equation}
		Induction along the chosen paths gives
		\[
		F_{i_{\ell}i_{\ell}}(a)=\delta_\ell^{-1}\sigma_C(a)\delta_\ell,
		\qquad a\in\DD,\ 0 \le \ell \le m.
		\]
		Indeed, suppose the induction statement holds at $i_\ell=p$, and write
		$i_{\ell+1}=q$. If $p\preceq q$, then
		\[
		F_{qq}(a)
		\stackrel{\eqref{eq:diag-intertwining-preserving}}=
		\lambda_{pq}^{-1}F_{pp}(a)\lambda_{pq}
		=
		(\delta_\ell\lambda_{pq})^{-1}\sigma_C(a)(\delta_\ell\lambda_{pq});
		\]
		if $q\preceq p$, then
		\[
		F_{qq}(a)
		\stackrel{\eqref{eq:diag-intertwining-preserving}}=
		\lambda_{qp}F_{pp}(a)\lambda_{qp}^{-1}
		=
		(\delta_\ell\lambda_{qp}^{-1})^{-1}\sigma_C(a)
		(\delta_\ell\lambda_{qp}^{-1}).
		\]
		These are exactly the two recursive definitions of $\delta_{\ell+1}$. For
		$\ell=m$ we therefore obtain
		\begin{equation}\label{eq:induction-preserving}
			F_{ii}(a)=d_i^{-1}\sigma_C(a)d_i,
			\qquad a\in\DD,\ i\in C .
		\end{equation}
		Set
		\[
		D_C:=\sum_{i\in C}d_iE_{ii}\in P_CM_n(\DD)P_C.
		\]
		Then $D_C$ is invertible in $P_CM_n(\DD)P_C$. Define
		\[
		g_C:\mathop{\preceq_C}\to\DD^\times,
		\qquad
		g_C(i,j):=d_i\lambda_{ij}d_j^{-1}.
		\]
		We prove that $g_C$ is a transitive $\Cent(\sigma_C)^\times$-valued multiplier. If
		$i\preceq j\preceq k$ and $i,j,k$ are distinct, then
		Lemma~\ref{lem:chain-orientation} gives, in the preserving case, $\lambda_{ik}=\lambda_{ij}\lambda_{jk}$. Hence 
		\[
		g_C(i,k)=g_C(i,j)g_C(j,k).
		\]
		The cases $i=j$ and $j=k$ are immediate from $\lambda_{ii}=1$. If $i=k$ and $i\asymp j$, with $i\ne j$, then \eqref{eq:square-after-add} applied to $E_{ij}+E_{ji}$ gives $\lambda_{ij}\lambda_{ji}=1=\lambda_{ji}\lambda_{ij}$, and therefore 
		\[
		g_C(i,j)g_C(j,i)=1=g_C(i,i).
		\]
		Thus $g_C$ is transitive.
		
		For $i\preceq j$ and $a\in\DD$, we have
		\[
		\sigma_C(a)g_C(i,j)
		\stackrel{\eqref{eq:induction-preserving}}=
		d_iF_{ii}(a)\lambda_{ij}d_j^{-1}
		\stackrel{\eqref{eq:diag-intertwining-preserving}}=
		d_i\lambda_{ij}F_{jj}(a)d_j^{-1}
		\stackrel{\eqref{eq:induction-preserving}}=
		g_C(i,j)\sigma_C(a).
		\]
		Thus $g_C(i,j)\in\Cent(\sigma_C)^\times$.
		
		Finally, if $i\preceq j$, $i\ne j$, then
		\begin{align*}
			D_C\Phi(aE_{ij})D_C^{-1}
			&\stackrel{\eqref{eq:pure-preserving}}=
			d_iF_{ij}(a)d_j^{-1}E_{ij}
			\stackrel{\eqref{eq:sandwich-preserving}}=
			d_iF_{ii}(a)\lambda_{ij}d_j^{-1}E_{ij}
			\stackrel{\eqref{eq:induction-preserving}}=
			\sigma_C(a)g_C(i,j)E_{ij}\\
			& = g_C(i,j)\sigma_C(a)E_{ij}.
		\end{align*}
		For $i=j$, \eqref{eq:induction-preserving} gives
		\[
		D_C\Phi(aE_{ii})D_C^{-1}
		=
		\sigma_C(a)E_{ii}
		=
		g_C(i,i)\sigma_C(a)E_{ii}.
		\]
		By additivity,
		\[
		D_C\Phi(X)D_C^{-1}=\Theta^M_{\sigma_C,g_C}(X),
		\qquad X\in P_C\R P_C.
		\]
		
		It remains to treat the reversing case. The proof is the same, with the order
		of coefficients reversed. First, every $F_{ii}$ is an injective
		anti-endomorphism. Indeed, if $i\preceq j$, then
		\[
		\lambda_{ij}F_{ii}(ab)
		\stackrel{\eqref{eq:sandwich-reversing}}=
		F_{ij}(ab)
		\stackrel{\eqref{eq:sandwich-reversing}}=
		F_{ij}(b)F_{ii}(a)
		\stackrel{\eqref{eq:sandwich-reversing}}=
		\lambda_{ij}F_{ii}(b)F_{ii}(a),
		\]
		and left cancellation gives $F_{ii}(ab)=F_{ii}(b)F_{ii}(a)$. If
		$j\preceq i$, the coordinate $(j,i)$ gives
		\[
		F_{ii}(ab)\lambda_{ji}
		\stackrel{\eqref{eq:sandwich-reversing}}=
		F_{ji}(ab)
		\stackrel{\eqref{eq:sandwich-reversing}}=
		F_{ii}(b)F_{ji}(a)
		\stackrel{\eqref{eq:sandwich-reversing}}=
		F_{ii}(b)F_{ii}(a)\lambda_{ji},
		\]
		and right cancellation gives the same conclusion.
		
		Fix $r\in C$ and put $\tau_C:=F_{rr}$. For $i \in C$, choose a path $r=i_0, i_1, \ldots, i_m = i$ in $\Gamma_{\preceq}$ as above,
		define $\delta_0:=1$, and continue recursively by
		\[
		\delta_{\ell+1}
		:=
		\begin{cases}
			\delta_\ell\lambda_{i_\ell i_{\ell+1}}^{-1},
			&\text{if } i_\ell\preceq i_{\ell+1},\\[1mm]
			\delta_\ell\lambda_{i_{\ell+1}i_\ell},
			&\text{if } i_{\ell+1}\preceq i_\ell .
		\end{cases}
		\]
		Set again $d_i:=\delta_m$. For $p,q\in C$ with $p\preceq q$,
		\eqref{eq:sandwich-reversing} gives
		\begin{equation}\label{eq:diag-intertwining-reversing}
			F_{qq}(a)\lambda_{pq}=\lambda_{pq}F_{pp}(a),
			\qquad a\in\DD .
		\end{equation}
		The same path induction gives
		\begin{equation}\label{eq:induction-reversing}
			F_{ii}(a)=d_i^{-1}\tau_C(a)d_i,
			\qquad a\in\DD,\ i\in C .
		\end{equation}
		For a forward step $p\preceq q$ this follows from $F_{qq}(a)=\lambda_{pq}F_{pp}(a)\lambda_{pq}^{-1}$, and for a reverse step
		$q\preceq p$ from $F_{qq}(a)=\lambda_{qp}^{-1}F_{pp}(a)\lambda_{qp}$; these are exactly the two recursive definitions above.
		
		Set again $D_C:=\sum_{i\in C}d_iE_{ii}$, and define
		\[
		h_C:\mathop{\preceq_C}\to\DD^\times, \qquad h_C(i,j):=d_i\lambda_{ij}^{-1}d_j^{-1}.
		\]
		For $i\preceq j$ and $a\in\DD$, we have
		\[
		\tau_C(a)h_C(i,j)
		\stackrel{\eqref{eq:induction-reversing}}=
		d_iF_{ii}(a)\lambda_{ij}^{-1}d_j^{-1}
		\stackrel{\eqref{eq:diag-intertwining-reversing}}=
		d_i\lambda_{ij}^{-1}F_{jj}(a)d_j^{-1}
		\stackrel{\eqref{eq:induction-reversing}}=
		h_C(i,j)\tau_C(a).
		\]
		Thus $h_C(i,j)\in\Cent(\tau_C)^\times$.
		
		The same verification as above, now using the reversing identity
		$\lambda_{ik}=\lambda_{jk}\lambda_{ij}$ from Lemma~\ref{lem:chain-orientation},
		shows that $h_C$ is transitive on $\preceq_C$. Finally, for $i\prec j$,
		\[
		D_C\Phi(aE_{ij})D_C^{-1}
		\stackrel{\eqref{eq:pure-reversing}}=
		d_jF_{ij}(a)d_i^{-1}E_{ji}
		\stackrel{\eqref{eq:sandwich-reversing}}=
		d_j\lambda_{ij}F_{ii}(a)d_i^{-1}E_{ji}
		\stackrel{\eqref{eq:induction-reversing}}=
		h_C(i,j)^{-1}\tau_C(a)E_{ji}.
		\]
		For $i=j$, \eqref{eq:induction-reversing} gives
		\[
		D_C\Phi(aE_{ii})D_C^{-1}=\tau_C(a)E_{ii}=
		h_C(i,i)^{-1}\tau_C(a)E_{ii}.
		\]
		By additivity,
		\[
		D_C\Phi(X)D_C^{-1}=\Theta^A_{\tau_C,h_C}(X),
		\qquad X\in P_C\R P_C.
		\]
		This proves the proposition.
	\end{proof}
	
	\begin{theorem}\label{thm:standard-form-additive}
		Let $\R_{\preceq}\subseteq M_n(\DD)$ be an SMR with $|\DD|>2$, and let
		$\phi:\R_{\preceq}\to M_n(\DD)$ be an injective additive Jordan semi-triple
		map. Then there exist $S\in\GL_n(\DD)$, signs
		$\eps_C\in\{1,-1\}\subseteq Z(\DD)$ and maps $\Theta_C : P_C \R P_C \to P_C M_n(\DD) P_C$, indexed by the connected components
		$C$ of $\Gamma_{\preceq}$, such that
		\begin{equation}\label{eq:main-normal-form}
			\phi(X)
			=
			S^{-1}\left(\bigoplus_C\eps_C\Theta_C(P_CXP_C)\right)S,
			\qquad X\in\R_{\preceq}.
		\end{equation}
		Moreover, for each connected component $C$, the map $\Theta_C$ is either
		$\Theta^M_{\sigma_C,g_C}$ as in \eqref{eq:Theta-sigma}, with
		$\sigma_C$ an injective endomorphism of $\DD$ and
		$g_C:\mathop{\preceq_C}\to\Cent(\sigma_C)^\times$ transitive, or
		$\Theta^A_{\tau_C,h_C}$ as in \eqref{eq:Theta-tau}, with
		$\tau_C$ an injective anti-endomorphism of $\DD$ and
		$h_C:\mathop{\preceq_C}\to\Cent(\tau_C)^\times$ transitive.
		
		Conversely, any map of the above form is an injective Jordan semi-triple map.
	\end{theorem}
	
	\begin{proof}
		Normalize $\phi$ as in Lemma~\ref{lem:normalization}, and denote the
		normalized map by $\Phi$. We use the standing convention after
		Proposition~\ref{prop:diagonal-straightening}.
		
		On nonsingleton connected components,
		Proposition~\ref{prop:normalized-standard} gives the asserted standard form.
		On a singleton component $C=\{i\}$, the restriction has the form
		$\Phi(aE_{ii})=\theta_C(a)E_{ii}$, for some nonzero additive scalar Jordan semi-triple map $\theta_C : \DD\to \DD$. Since $\Phi(E_{ii}) = E_{ii}$,  Remark~\ref{rem:additive-scalar-JST}
		shows that $\theta_C$ is an injective endomorphism or anti-endomorphism of $\DD$. Combining the componentwise diagonal similarities, we obtain
		$T\in\GL_n(\DD)$ such that
		\begin{equation}\label{eq:displayed-image}
			T\Phi(X)T^{-1}
			=
			\bigoplus_C\Theta_C(P_CXP_C),
			\qquad X\in\R_{\preceq}.
		\end{equation}
		
		It remains to undo the normalization. Let $J$ denote the corresponding conjugate of $\phi(I)$ after the fixed
		range similarities already absorbed into the current normalized map $\Phi$.
		Then $J^2=I$ and $J$ centralizes $\Phi(\R_{\preceq})$. Therefore,
		$\widehat J:=TJT^{-1}$ centralizes the image of the map from \eqref{eq:displayed-image}. This image contains all diagonal
		idempotents, so $\widehat J$ is diagonal. Since $\widehat J^2=I$, its diagonal
		entries are $\pm1$. For each $i\prec j$, the map $X\mapsto T\Phi(X)T^{-1}$ sends $E_{ij}$ to a nonzero scalar multiple of either $E_{ij}$ or $E_{ji}$. Thus commutation with
		$\widehat J$ forces the corresponding diagonal signs to be equal. Hence the
		signs are constant on connected components, and
		\[
		\widehat J=\bigoplus_C\eps_CP_C,
		\qquad \eps_C\in\{1,-1\}.
		\]
		Undoing the initial normalization and the fixed range similarities, we get
		\[
		S\phi(X)S^{-1}
		=
		\widehat J\left(\bigoplus_C\Theta_C(P_CXP_C)\right)
		=
		\bigoplus_C\eps_C\Theta_C(P_CXP_C)
		\]
		for some fixed $S\in\GL_n(\DD)$. This is precisely
		\eqref{eq:main-normal-form}.
		
		The converse verification is straightforward. As the maps
		$\Theta^M_{\sigma_C,g_C}$ and $\Theta^A_{\tau_C,h_C}$ are respectively
		monomorphisms and anti-monomorphisms, each component map is a
		Jordan semi-triple map. The block direct sum, multiplication by the central
		signs $\eps_C=\pm1$, and conjugation by $S$ preserve the Jordan semi-triple
		identity. Injectivity follows from the injectivity of the component maps.
	\end{proof}
	
	\section*{Acknowledgments}
	This research was supported by the European Union -- NextGenerationEU through the National Recovery and Resilience Plan 2021--2026 Institutional grants of University of Zagreb Faculty of Science (IK IA 1.1.3. Impact4Math, PMF-CROFUND).
	
	The authors disclose that OpenAI's GPT-5.5 was used for copyediting and language improvement of the manuscript.

\end{document}